\documentclass[11pt]{amsart}

\usepackage{graphicx}        
\usepackage{multicol}        
\usepackage[bottom]{footmisc}
\usepackage{amscd,amsmath,amssymb,amsfonts}
\usepackage[cmtip, all]{xy}

\newtheorem{theorem}{Theorem}[section]
\newtheorem{proposition}[theorem]{Proposition}
\newtheorem{lemma}[theorem]{Lemma}
\newtheorem{corollary}[theorem]{Corollary}

\theoremstyle{definition}
\newtheorem{definition}[theorem]{Definition}
\theoremstyle{remark}
\newtheorem{remark}[theorem]{Remark}
\newtheorem{remarks}[theorem]{Remarks}
\newtheorem{ex}[theorem]{Example}

\numberwithin{equation}{section}

\newcommand{\CH}{{\rm CH}}

\newcommand{\red}{{\rm red}}

\newcommand{\End}{{\rm End}}
\newcommand{\Hom}{{\rm Hom}}

\newcommand{\Spec}{{\rm Spec\,}}

\newcommand{\Tr}{{\rm Tr}}

\newcommand{\0}{\emptyset}

\newcommand{\sC}{{\mathcal C}}

\newcommand{\sE}{{\mathcal E}}
\newcommand{\sF}{{\mathcal F}}
\newcommand{\sG}{{\mathcal G}}
\newcommand{\sH}{{\mathcal H}}
\newcommand{\sI}{{\mathcal I}}
\newcommand{\sJ}{{\mathcal J}}
\newcommand{\sK}{{\mathcal K}}
\newcommand{\sL}{{\mathcal L}}

\newcommand{\sN}{{\mathcal N}}
\newcommand{\sO}{{\mathcal O}}

\newcommand{\sV}{{\mathcal V}}

\newcommand{\sX}{{\mathcal X}}
\newcommand{\sY}{{\mathcal Y}}

\newcommand{\A}{{\mathbb A}}

\newcommand{\G}{{\mathbb G}}

\renewcommand{\L}{{\mathbb L}}

\renewcommand{\P}{{\mathbb P}}

\newcommand{\V}{{\mathbb V}}

\newcommand{\Z}{{\mathbb Z}}

\newcommand{\BM}{{\operatorname{B.M.}}}

\renewcommand{\det}{\operatorname{det}}

\newcommand{\id}{{\operatorname{\rm Id}}}
\newcommand{\Zar}{{\text{\rm Zar}}}

\newcommand{\Sch}{{\operatorname{\mathbf{Sch}}}}

\newcommand{\op}{{\text{\rm op}}}
\newcommand{\<}{\langle}
\renewcommand{\>}{\rangle}

\newcommand{\Div}{{\operatorname{div}}} 

\newcommand{\Coh}{\operatorname{Coh}}
\newcommand{\QCoh}{\operatorname{QCoh}}

\newcommand{\del}{\partial}

\newcommand{\Sm}{{\mathbf{Sm}}}

\newcommand{\Proj}{{\operatorname{Proj}}}

\newcommand{\Sym}{{\operatorname{Sym}}}

\newcommand{\Ext}{{\operatorname{Ext}}}

\newcommand{\Tor}{{\operatorname{\rm Tor}}}

\newcommand{\bTr}{\mathbf{Tr}}
\newcommand{\mfC}{\mathfrak{C}}

\newcommand{\MGL}{{\operatorname{MGL}}}
\newcommand{\KGL}{{\operatorname{KGL}}}
\newcommand{\MSL}{{\operatorname{MSL}}}

\newcommand{\GW}{{\operatorname{GW}}} 
 
\newcommand{\SH}{{\operatorname{SH}}} 
\newcommand{\Th}{{\operatorname{Th}}} 
\newcommand{\Cst}{\mathfrak{C}^{st}} 
 
\newcommand{\sHom}{\mathcal{H}om}
 
\newcommand{\Aut}{{\operatorname{Aut}}}

\newcommand{\SL}{\operatorname{SL}}
\newcommand{\BSL}{\operatorname{BSL}}

\newcommand{\Deg}{\text{deg}}
\newcommand{\ev}{\text{\it ev}}
\newcommand{\vir}{\text{\it vir}}
\newcommand{\perf}{\text{\it perf}}

\newcommand{\ind}[1]{}

\newcommand{\inp}[1]{}

\newcommand{\mS}{\mathbb{S}}

\makeatletter
\@namedef{subjclassname@2020}{\textup{2020} Mathematics Subject Classification}
\makeatother

\begin{document}
\setcounter{tocdepth}{1}

\title{The intrinsic stable normal cone}
\author{Marc Levine}
\email{marc.levine@uni-due.de}
\address{Fakult\"at Mathematik\\
Universit\"at Duisburg-Essen\\
Thea-Leymann-Stra{\ss}e 9\\
45127 Essen\\
Germany}
\thanks{The author thanks the DFG for support through the grant LE 2259/8-1 and the ERC  for support through the project QUADAG. This paper is part of a project that has received funding from the European Research Council (ERC) under the European Union's Horizon 2020 research and innovation programme (grant agreement No. 832833).}

\begin{abstract} We construct an analog of the intrinsic normal cone of Behrend-Fantechi  in the setting of motivic stable homotopy theory. A perfect obstruction theory gives rise to a virtual fundamental class in $\sE$-cohomology for any motivic cohomology theory $\sE$; this includes the oriented Chow groups of Barge-Morel  and Fasel. 
\end{abstract}

\maketitle

\tableofcontents

\section{Introduction}\label{sec:Intro}
The various versions of modern enumerative geometry, including Gromov-Witten theory and Donaldson-Thomas theory, are based on two important constructions due to Behrend and Fantechi \cite{BF}. The first is the construction of the intrinsic normal cone $\mfC_Z$ of a Deligne-Mumford stack $Z$ over some base-scheme $B$. The second, based on the first, is the virtual fundamental class $[Z,[\phi]]^\vir\in \CH_r(Z)$ associated to a perfect obstruction theory $[\phi]:E_\bullet\to \L_{Z/B}$ on $Z$, with $r$ the virtual rank of $E_\bullet$.  In case $r=0$ and $Z$ is proper over a field $k$, one has the numerical invariant $\Deg_k[Z,[\phi]]^\vir$; more generally, one can cut down $[Z,[\phi]]^\vir$ to dimension zero by taking so-called descendants, and then taking the degree of the resulting 0-cycle. 

A perfect obstruction theory on $Z$ is given by a map $[\phi]:E_\bullet\to \L_{Z/B}$ in $D^\perf(Z)$ such that $E_\bullet$ is locally represented on $Z$ by a two-term complex $F_1\to F_0$ in degrees $0, 1$ (we use homological notation) and such that the map $[\phi]$ induces an isomorphism on the sheaf $h_0$ and a surjection on $h_1$.

In case $[\phi]$ admits a global resolution $(F_1\to F_0)\to\L_{Z/B}$,  the virtual fundamental class is defined by embedding $\mfC_Z$ in the quotient stack $[F^1/F^0]$ ($F^i:=\V(F_i)$), pulling back $\mfC_Z$ via the quotient map $F^1\to [F^1/F^0]$, which gives  the subcone $\mfC(F_\bullet)\subset F^1$,  and then intersecting with the zero-section:
\[
[Z,[\phi]]^\vir:=0_{F^1}^!([\mfC(F_\bullet)]).
\]
Here $F^i\to Z$ is the vector bundle dual to $F_i$ and $[\mfC(F_\bullet)]$  is the fundamental class associated to the closed subscheme $\mfC(F_\bullet)$ of $F^1$. If one wishes to extend this type of construction to more general cohomology theories, there may be a problem in even defining the fundamental class $[\mfC(F_\bullet)]$. For instance, in algebraic cobordism $\Omega_*$, fundamental classes of arbitrary schemes do not exist \cite[\S3]{LevineFund}.  

The main point of this paper is to reinterpret the constructions of the intrinsic normal cone, its fundamental class, and the virtual fundamental class associated to a perfect obstruction theory in the setting of motivic homotopy theory. Rather than taking a DM or Artin stack as our basic object,  we work in the $G$-equivariant setting, following the current state of the art  in motivic stable homotopy theory, for which unfortunately a suitable theory for stacks is not yet available.  We will assume that $G$ is {\em tame} in the sense of \cite{Hoyois6}; this includes the case of  a split torus,  a finite \'etale group scheme of order prime to all residue characteristics, or a reductive group scheme in characteristic zero. For the full theory, we will also need to assume that the base-scheme $B$ is affine and the $G$-scheme $Z$ carrying the perfect obstruction theory is $G$-quasi-projective over $B$.

In spite of these restrictions, we gain a great deal of generality. We construct an ``intrinsic stable normal cone'' $\Cst_Z$ for each $G$-quasi-projective $B$-scheme $Z$, with $\Cst_Z$ defined as an object in the equivariant motivic stable homotopy category $\SH^G(B)$ (see Theorem~\ref{thm:IntStableNormalCone} and Definition~\ref{Def:IntStableNormalCone}). $\Cst_Z$ carries a fundamental class $[\Cst_Z]$ in co-homotopy $\mS_B^{0,0}(\Cst_Z)$ (Definition~\ref{def:FundClass}). Moreover, for a perfect obstruction theory $[\phi]:E_\bullet\to \L_{Z/B}$ on $Z$, we use $[\Cst_Z]$ to construct a virtual fundamental class in twisted Borel-Moore homology (see Definition~\ref{def:BMHomology})
\[
[Z, [\phi]]^\vir\in \mS_B^\BM(Z, \V(E_\bullet)).
\]
 As we are not relying on a theory in the setting of stacks, the construction (Definition~\ref{def:VirtFundClass}) relies on a number of choices; Proposition~\ref{prop:VirtFundClass} provides the crucial independence of these choices.

If $\sE\in \SH^G(B)$ is a motivic ring spectrum (i.e., a monoid object in $\SH^G(B)$) with unit map $\epsilon_\sE:\mS_B\to\sE$, applying $\epsilon_\sE$ to $[\Cst_Z]$ or $[Z, [\phi]]^\vir$ gives us elements
\begin{align*}
&[\Cst_Z]_\sE\in \sE^{0,0}(\Cst_Z);\\
&[Z, [\phi]]^\vir_\sE\in \sE^\BM(Z, \V(E_\bullet)).
\end{align*}

For simplicity,  we consider the case $G=\{\id\}$.  If we take $\sE=H\Z$, the spectrum representing motivic cohomology, then, suitably interpreted, these classes reduce to the classes defined by Behrend-Fantechi. More generally, if $\sE$ is orientable, then 
\[
\sE^\BM(Z, \V(E_\bullet))\cong \sE^\BM_{2r,r}(Z)
\]
with $r$ the virtual rank of $E_\bullet$.  We can thus identify $\sE^{0,0}(\Cst_Z)$ and $\sE^\BM(Z, \V(E_\bullet))$ as Borel-Moore $\sE$-homology, giving classes
\begin{align*}
&[\Cst_Z]_\sE\in \sE^\BM_{2d_M,d_M}(\mfC_i);\\
&[Z, [\phi]]^\vir_\sE\in \sE^\BM_{2r,r}(Z).
\end{align*}
Here $\mfC_i$ is the normal cone of $Z$ for a given closed immersion $i:Z\to M$ with $M$ smooth over $B$, $d_M$ is the dimension of $M$ over $B$ (which is the same as the dimension of $\mfC_i$ over $B$) and $r$ is the virtual rank of $E_\bullet$. Besides motivic cohomology, this includes such oriented theories such as (homotopy invariant) algebraic $K$-theory or algebraic cobordism.

If we work with theories $\sE$ that are not oriented, the identification of the group carrying the virtual fundamental class becomes more complicated. However, there is an interesting class of theories, the $\SL$-oriented theories, which admit a Thom isomorphism for bundles with a trivial determinant. One such theory is   cohomology in the sheaf of Milnor-Witt $K$-groups (see \cite{MorelA1}). The part of this theory corresponding to the Chow groups gives the Barge-Morel theory of {\em oriented} Chow groups (or Chow-Witt groups)
\[
\widetilde{\CH}^n(X):=H^n(X, \sK^{MW}_n)
\]
a formula reminiscent of Bloch's formula relating the classical Chow groups with Milnor $K$-theory. There are also twisted versions of the oriented Chow groups
\[
\widetilde{\CH}^n(X; L):=H^n(X, \sK^{MW}_n(L))
\]
for a line bundle $L$ on $X$.  These formulas for the oriented Chow groups are  only valid for smooth $X$, but one has a straightforward extension to the general case using Borel-Moore homology. The general theory gives us classes
\begin{align*}
&[\Cst_Z]_{\sK^{MW}_*}\in \widetilde{\CH}_{d_M}(\mfC_i; i^*\omega^{-1}_{M/B});\\
&[Z, [\phi]]^\vir_{\sK^{MW}_*}\in \widetilde{\CH}_r(Z; \det E_\bullet).
\end{align*}

In this setting  the push-forward maps on the oriented Chow groups are restricted to {\em oriented} proper morphisms. This still allows one to achieve a refinement of the usual Gromov-Witten type invariants in case the given perfect deformation theory $E_\bullet$ not only has virtual rank zero, but also has trivial virtual determinant bundle  modulo squares. In this case, we have 
\[
\Deg([Z, [\phi]]^\vir_{\sK^{MW}_*})\in \GW(k)
\]
where $\GW(k)$ is the Grothendieck-Witt group of the base-field $k$; applying the rank homomorphism $\GW(k)\to \Z$ recovers the classical degree. We hope that this approach will prove useful, for example, in studying the enumerative geometry of real varieties, where one takes the signature rather than the rank to obtain the relevant invariant.

Our approach is essentially formal:  our construction uses three ingredients beyond some elementary geometry of normal cones, namely:
\begin{enumerate}
\item The existence of Grothendieck's six operations for the equivariant motivic stable homotopy category $\SH^G(-):(\Sch^{G}/B)^\op\to \bTr$. Here $\bTr$ is the 2-category of triangulated categories. In particular, for each $G$-vector bundle $V\to X$, we have the autoequivalence $\Sigma^V:\SH^G(X)\to \SH^G(X)$. 
\item For $X\in \Sch^G/B$, we have the path groupoid $\sV^G(X)$ of the $G$-equivariant $K$-theory space of $X$. We need the existence of a natural transformation $\Sigma^{-}:\sV^G(-)\to \Aut(\SH^G(-))$ extending the map  $V\mapsto \Sigma^V$,   such that the exceptional push-forward and pull-back for a smooth morphism $f:X\to Y$ is given by $f^!=\Sigma^{T_{X/Y}}\circ f^*$, $f_!=f_\#\circ \Sigma^{-T_{X/Y}}$, where $f_\#$ is the left adjoint to $f^*$. 
\item $\A^1$-homotopy invariance: for $p:V\to Z$ an affine space bundle, co-unit of adjunction $p_!p^!\to \id_{\SH^G(Z)}$ is an isomorphism.
\end{enumerate}
Presumably other contravariant functors $\Sch^G/B\to \bTr$ have these three properties. 

A construction of the fundamental class of the normal cone $\mfC_{Z\subset M}$ in algebraic cobordism was communicated to us by Parker Lowrey some years ago. Our construction of the fundamental class may be viewed as a generalization of this method, see Example~\ref{ex:ClassicalExamples} for further details. F. D\'eglise, F. Jin and A. Khan \cite{DJK}
 have generalized aspects of the work of Lowrey-Sch\"urg \cite{LowreySchuerg},  constructing fundamental classes of quasi-smooth derived schemes in a motivic stable homotopy category of derived schemes; we expect there is a suitable dictionary translating between some of their constructions and some of the ones given here. For instance, A. Khan \cite{KhanDerVir} has constructed a virtual fundamental class associated to a quasi-smooth morphism $f:\sX\to \sY$ of derived higher stacks. Due to his use of (higher) stacks, Khan's construction takes place in the framework of the motivic stable homotopy category for the \'etale topology, which places the virtual fundamental class in a Borel-Moore homology for a theory that satisfies \'etale descent and thus is not applicable for getting invariants in more general theories, such as those closely related to the Grothendieck-Witt ring. 
 
It seems one can restrict Khan's construction to the setting of derived schemes, for which one does have a good theory comparable to that of $\SH(B)$ for $B$ a scheme, in fact a central result  of Khan \cite{KhanDerived} states that  for $X$ the underlying classical scheme of a (connective) derived scheme $\sX$, the restriction map $\SH(\sX)\to \SH(X)$ is an equivalence. A quasi-smooth map of derived schemes $f:\sX\to \sY$ does give rise to a (relative) perfect obstruction theory on the underlying classical scheme of $\sX$, and we expect that Khan's classes agree with the ones constructed here. However, Sch\"urg \cite{Schurg} has found obstructions for a given perfect obstruction theory to arise this way, so the approach using derived schemes may not always be applicable. These issues of relating the derived theory with the one presented here is being investigated by A. D'Angelo \cite{DAngelo}.
 
In \S\ref{sec:Background} we recall the necessary background from motivic homotopy theory. We 
construct the intrinsic stable normal cone  in \S\ref{sec:NormalThom} and its   fundamental class in \S\ref{sec:FundClass}. The  construction of the virtual fundamental class in a special case (for a  {\em reduced, normalized representative} of a perfect obstruction theory)  is given in \S\ref{sec:RedNormVirClass}, where the formula is completely analogous to the one of Behrend-Fantechi. 

In \S\ref{sec:ObstThy}, we show how a Jouanolou cover $p_Z:\tilde{Z}\to Z$ and a perfect obstruction theory $[\phi]:E_\bullet\to \L_{Z/B}$ give rise to an {\em induced} perfect obstruction theory $p_Z^![\phi]:p_Z^!E_\bullet\to \L_{\tilde{Z}/B}$, which admits a reduced, normalized representative. 

 The general case is handled in 
\S\ref{sec:VFCGeneral}.  Roughly speaking,  one uses Jouanolou's trick and $\A^1$-homotopy invariance properties and the results of \S\ref{sec:ObstThy} to reduce the general case to the case  handled in \S\ref{sec:RedNormVirClass}.  The main point is to show that the resulting class is independent of the various choices made along the way.

 We compare our constructions with those of Behrend-Fantechi and discuss variants in \S\ref{sec:Comp}. We conclude with some explicit computations in \S\ref{sec:LocalDeg}, looking at critical loci and relating the virtual fundamental class of a local complete intersection to the $\A^1$ local degree defined in \cite{KW}.

I would like to thank the referee for a series of insightful comments, which were used to improve an earlier version of this paper. Further improvements were made possible by detailed comments and suggestions from Sabrina Pauli, for which I am very grateful.

\section{Background on motivic homotopy theory}\label{sec:Background}
We begin by recalling some of the aspects of the six operations for the motivic stable homotopy category. We refer the reader to \cite{Ayoub, CD, JardineMotSym, MorelVoev, VoevICM} for details on the non-equivariant case and \cite{Hoyois6} for the extension to the equivariant setting. 

Fix a noetherian affine scheme $U$ with a flat, finitely presented linearly reductive group scheme $G_0$ over $U$  (see \cite[Definition 2.14]{Hoyois6}).  

We fix a quasi-projective $U$-scheme $B\to U$ (with trivial $G_0$-action) as base-scheme  and let $G=G_0\times_UB$. A $G$-equivariant morphism $q:Y\to X$ of $G$-schemes over $B$ is called {\em $G$-quasi-projective} if there is a $G$-vector bundle $V\to X$ and a $G$-equivariant locally closed immersion $i:Y\to \P(V)$ of $X$-schemes.  We let   $\Sch^G/B$ be the full subcategory of $G$-schemes over $B$ with objects the $G$-quasi-projective  $B$-schemes and let $\Sm^G/B$ be the full subcategory of smooth $B$-schemes in $\Sch^G/B$. We will usually denote the structure morphism for $Z\in \Sch^G/B$ by $\pi_Z:Z\to B$.

For $X\in \Sch^G/B$, we have the category $\QCoh_X^G$ of quasi-coherent $\sO_X$-modules with $G$-action and the full subcategory $\Coh^G_X$ of coherent sheaves. We call $\sF$ in $\Coh_X^G$ locally free if $\sF$ is locally free as an $\sO_X$-module. We let $D^b_G(X)$ denote the bounded derived category of coherent $G$-sheaves, $D_G(X)$ the unbounded derived category of quasi-coherent $G$-sheaves and $D^\perf_G(X)$ the full subcategory of  $D_G(X)$ of complexes isomorphic in $D_G(X)$ to a bounded complex of locally free coherent sheaves. Such a complex is called a {\em perfect} complex. 

 We will use homological notation for complexes: for a homological complex $C_\bullet$, $\tau_{\le n}C_\bullet$ is the quotient complex which is $C_m$ in degree $m<n$, $0$ in degree $m>n$ and $C_n/\del(C_{n+1})$ in degree $n$. 

We will assume that $U$ has the {\em $G_0$-resolution property}, namely, that each $\sF\in \Coh^{G_0}_U$ admits a surjection $\sE\to \sF$ from a locally free   $\sE$ in $\Coh^{G_0}_U$ (see \cite[Definition 2.7]{Hoyois6}). This implies that the group scheme $G$ over $B$ is {\em tame} in the sense of \cite[Definition 2.26]{Hoyois6}. 
 Examples  of linearly reductive $G_0$ such that $U$ has the $G_0$-resolution property include
\begin{itemize}
\item   $G_0$ is finite locally free of order invertible on $U$. 
\item $G_0$ is of multiplicative type and is isotrivial.
\item $U$ has characteristic zero and $G_0$ is reductive with isotrivial radical and coradical (e.g., $G_0$ is semisimple).
\end{itemize}
See \cite[Examples 2.8, 2.16, 2.27]{Hoyois6}.

Hoyois shows \cite[Lemma 2.11]{Hoyois6} that each $Z\in \Sch^G/B$ has the  $G$-resolution property. In addition, if $Z\in \Sch^G/B$ is affine, then  a locally free coherent $G$-sheaf $\sF$  on $Z$ is projective in $\QCoh_X^G$ \cite[Lemma 2.17]{Hoyois6}. This also implies that a complex $E_\bullet$ in $D_G(Z)$ that is locally (on $Z_\Zar$)  a perfect complex is in fact a perfect complex on $Z$. Similarly, if $E_\bullet\in D_G(X)$ is locally isomorphic to a complex of locally free coherent sheaves that is 0 in degrees outside a given interval $[a,b]$, then $E_\bullet$ is isomorphic to a complex of  coherent locally free $G$-sheaves on $X$ that is 0 in degrees outside $[a,b]$. Such a complex is called a {\em perfect complex supported in $[a,b]$}.

For $E\in \Coh_X^G$ locally free, we have the associated vector bundle $p:\V(E)\to X$, with
$\V(E):=\Spec_{\sO_X}\Sym^* E$.
The $G$-action on $E$  gives $\V(E)$ a $G$-action, with  $p:\V(E)\to X$ a $G$-equivariant morphism.

We will  often drop the ``$G$'' in our notations, speaking of $B$-morphisms for  $G$-equivariant $B$-morphisms, vector bundles $V\to X$ for  $G$-vector bundles, etc.

Let $\bTr$ be the 2-category of  triangulated categories. Following \cite[\S6, Theorem 6.18]{Hoyois6}, we have the motivic stable homotopy category 
\[
\SH^G(-):\Sch^G/B^\op\to \bTr; 
\]
for $f:Y\to X$ in $\Sch^G/B$, we have the exact functor $f^*:\SH^G(X)\to \SH^G(Y)$ with right adjoint $f_*:\SH^G(Y)\to \SH^G(X)$ and the exceptional pull-back $f^!:\SH^G(X)\to \SH^G(Y)$ with left adjoint $f_!:\SH^G(X)\to \SH^G(Y)$. If $f$ is a smooth morphism, $f^*$ admits the left adjoint $f_\#$.  $\SH^G(X)$ is a closed symmetric monoidal triangulated category with product denoted $\wedge_X$, unit $1_X$ (we often write $\mS_B$ for $1_B$) and internal Hom $\sHom_X(-,-)$; $f^*$ is a symmetric monoidal functor and $f_*$ and $f_!$ satisfy projection formulas, that is, for $f:Y\to X$, $f_*$ and $f_!$ are $\SH^G(X)$-module maps and the same holds for $f_\#$ if $f$ is smooth. There is a natural transformation $\eta^f_{!*}:f_!\to f_*$ which is an isomorphism if $f$ is proper. See also the earlier treatments  \cite{Ayoub} and \cite{CD} for the non-equivariant case. 

For the pair of adjoint functors $a_!\dashv a^!$, we let $e_a:a_!a^!\to \id$ denote the co-unit.
For the pair of adjoint functor $a^*\dashv a_*$, we let $u_a:\id\to a_*a^*$ denote the unit. 
We will use analogous notation for other adjoint pairs, leaving the context to make the meaning clear.

For $p:V\to X$ a vector bundle with zero-section $s:X\to V$, we have the $V$-suspension and $-V$-suspension operators $\Sigma^{\pm V}:\SH^G(X)\to \SH^G(X)$ defined as
\[
\Sigma^V:=p_\#\circ s_*=p_\#\circ s_!,\ \Sigma^{-V}:=s^!\circ p^*,
\]
with $\Sigma^V$ the left adjoint to $\Sigma^{-V}$. These endofunctors are in fact inverse equivalences \cite[Proposition 6.5]{Hoyois6}.

For a locally free sheaf $E\in\Coh_X^G$, we thus have  the autoequivalence $\Sigma^{\V(E)}$ of $\SH^G(X)$. Ayoub \cite[Th\'eor\`em 1.5.18]{Ayoub} and Riou \cite[Proposition 4.1.1]{RiouRR} have shown that for $G=\{\id\}$, the association $E\mapsto \Sigma^{\V(E)}$ extends to a functor 
\[
\Sigma^{-}:\sV(X)\to \Aut(\SH(X)),
\]
Here $\sV(X)$ is the path groupoid of the $K$-theory space of $X$  and $\Aut(\SH(X))$ is the category with objects the auto-equivalences of $\SH(X)$ and morphisms the natural isomorphisms.   

$\Sigma^{-}$ natural with respect to   $f^*$, $f^!$, $f_*$ and $f_!$ for morphisms $f:Y\to X$:
\[
\Sigma^{f^*\V(E)}\circ f^?\cong f^?\circ \Sigma^{\V(E)},\ f_?\circ \Sigma^{f^*\V(E)} \cong \Sigma^{\V(E)}\circ f_?,
\]
for $?=*, !$. This extends without problem to the equivariant case to give a functor $\Sigma^{-}:\sV^G(X)\to \Aut(\SH^G(X))$, with $\sV^G(X)$ the path groupoid of the $G$-equivariant $K$-theory space of $X$, having the properties listed above.

We write $\Sigma^{\V(E_\bullet)}$ for the image of a perfect complex $E_\bullet$ under this functor.\footnote{Many authors, for example \cite{Hoyois6}, use the notation $\Sigma^{E_\bullet}$ for our notation $\Sigma^{\V(E_\bullet)}$. For example, for $X$ smooth over $S$ and $E=\Omega_{X/S}$, we have $\V(E)=T_{X/S}$, the relative tangent bundle, and the operator denoted $\Sigma^{\Omega_{X/S}}$ in \cite{Hoyois6} will be written here as $\Sigma^{T_{X/S}}$}  For each distinguished triangle $E^1_\bullet\to E_\bullet\to E^2_\bullet\to$ in $D_G^\perf(X)$ there is an isomorphism
\[
 \Sigma^{\V(E_\bullet)}\cong \Sigma^{\V(E^{1}_\bullet)}\circ \Sigma^{\V(E^{2}_\bullet)},
\]
natural with respect to isomorphisms of distinguished triangles, 
and for each $E_\bullet$ an isomorphism $\Sigma^{\V(E_\bullet[1])}\cong (\Sigma^{\V(E_\bullet)})^{-1}$.  Thus, if $E_\bullet=(E_n\to\ldots\to E_m)$, supported in $[m,n]$, then $\Sigma^{\V(E_\bullet)}$ is canonically isomorphic to $\Sigma^{(-1)^n\V(E_n)}\circ\ldots\circ \Sigma^{(-1)^m\V(E_m)}$. 

 For $f:Y\to X$  smooth, there are canonical isomorphisms \cite[Theorem 6.9]{Hoyois6}
\[
f_!\cong f_\#\circ \Sigma^{-T_f}, \ f^!\cong \Sigma^{T_f}\circ f^*,
\]
with $T_f\to Y$ the relative tangent bundle, giving the canonical isomorphism 
\[
f_!f^!\cong f_\#f^*.
\]
If $f:V\to X$ is an affine space bundle over $X$, then the $\A^1$-homotopy property shows that the co-unit of the adjunction $f_\#\dashv f^*$, $e_f:f_\#f^*\to \id$, is a natural isomorphism.

Besides the equivariant stable motivic homotopy category, Hoyois has defined an equivariant unstable motivic category $\sH^G_\bullet(X)$ for $X\in\Sch^G/B$ \cite[\S5]{Hoyois6}, generalizing the constructions of Morel-Voevodsky \cite{MorelVoev} in the non-equivariant case. There is an infinite $T$-suspension functor $\Sigma^\infty_T:\sH^G_\bullet(X)\to \SH^G(X)$; we often simply write $\sX$ for $\Sigma^\infty_T\sX$ when the context makes the meaning clear. $\sH_\bullet^G(X)$ is a symmetric monoidal category with unit $S^0_X:=X_+$, and the unit $1_X\in \SH^G(X)$ is $\Sigma^\infty_T(S^0_X)$.

The functors $f^*$, $f_*$ are $T$-stabilizations of functors $f^*:\sH_\bullet(X)\to\sH_\bullet(Y)$, $f_*:\sH_\bullet(Y)\to \sH_\bullet(X)$, with $f^*$ left adjoint to $f_*$. If $f:Y\to X$ is smooth,  $f_\#$ is the $T$-stabilization of $f_\#:\sH_\bullet(Y)\to \sH_\bullet(X)$, left adjoint to $f^*$. Similarly, if $i:Y\to X$ is a closed immersion,   the maps $i_*=i_!:\SH^G(Y)\to \SH^G(X)$,   are the $T$-stabilizations of the unstable $i_*$.

To give the reader some intution about the suspension functors, we mention that for $V\to X$ a vector bundle with zero section  $s$, the suspension $\Sigma^V(1_X)$ is the stabilition of the {\em Thom space} $\Th_X(V):=p_\#s_*(1_V)\in\sH^G_\bullet(X)$,   and that $\Th_X(V)$ is canonically isomorphic to the cofiber of $V\setminus s(X)\to V$ in $\sH^G_\bullet(X)$: $\Th_X(V)\cong V/(V\setminus s(X))$. We will not be using these facts in the sequel.   

There are {\em exchange morphisms} associated to a cartesian diagram
\[
\xymatrix{
Z\ar[r]^-q\ar@{}@<-12pt>[r]_\Delta\ar[d]_g&Y\ar[d]^f\\
W\ar[r]_-p&X
}
\]
as follows:\\\\
1. We have $Ex(\Delta^*_*):p^*f_*\to g_*q^*$ defined 
as the composition
\[
p^*f_*\xrightarrow{u_g}g_*g^*p^*f_*=g_*q^*f^*f_*\xrightarrow{e_f}g_*q^*
\]
$Ex(\Delta^*_*)$ is an isomorphism if $p$ is smooth or if $f$ is proper.\\\\
2. Suppose that $p$ is smooth. The isomorphism $Ex(\Delta^*_\#):q_\#g^*\to f^*p_\#$ is defined as the composition
\[
q_\#g^*\xrightarrow{u_p}q_\#g^*p^*p_\#=q_\#q^*f^*p_\#\xrightarrow{e_q}f^*p_\#.
\]
3. Suppose $p$ is smooth. We have $Ex(\Delta_{\#*}):p_\#g_*\to f_*q_\#$ defined as the composition
\[
p_\#g_*\xrightarrow{u_f}f_*f^*p_\#g_*\xrightarrow{Ex(\Delta^*_\#)^{-1}}f_*q_\#g^*g_*
\xrightarrow{e_g}f_*q_\#.
\]
4. Suppose $p$ is smooth. We have the isomorphism $Ex(\Delta^{*!}):g^*p^!\to q^!f^*$ defined as the composition
\[
g^*p^!\cong g^*\Sigma^{T_{W/X}}p^*\cong \Sigma^{T_{Z/Y}}g^*p^*\cong  \Sigma^{T_{Z/Y}}q^*f^*\cong q^!f^*.
\]
5. Suppose $p$ is smooth. We have $Ex(\Delta_{!*}):p_!g_*\to f_*q_!$ defined  as the composition
\[
p_!g_*\cong p_\#\Sigma^{-T_{W/X}}g_*\cong p_\#g_*\Sigma^{-T_{Z/Y}}\xrightarrow{Ex(\Delta_{\#*})}f_*q_\#\Sigma^{-T_{Z/Y}}\cong f_*q_!
\]
6. For arbitrary $p$, we have the base-change  isomorphism $Ex(\Delta^*_!):p^* f_!\to g_!q^*$ (see \cite[Theorem 6.12]{Hoyois6}) satisfying $\eta^g_{!*}\circ Ex(\Delta^*_!)=Ex(\Delta^*_*)\circ \eta^f_{!*}$. If $f$ is smooth, combining $Ex(\Delta^*_!)$ with the naturality of $\Sigma^-$ gives the base-change isomorphism $Ex(\Delta^*_\#):p^*f_\#\to g_\#q^*$.

Suppose we have a closed immersion $i:Z\to X$ in $\Sch^G/B$, with open complement $j:U\to X$. This yields the {\em localization distinguished triangle} in $\SH^G(X)$
\begin{equation}\label{eqn:LocTria}
j_!j^!\xrightarrow{e_j} \id_{\SH^G(X)}\xrightarrow{u_i} i_*i^*\to j_!j^![1].
\end{equation}
Note that $j_!=j_\#$, $j^!=j^*$ and $i_*=i_!$.  

\begin{definition} Let $\pi_Z:Z\to B$ be in $\Sch^G/B$. The {\em Borel-Moore motive} of $Z$ over $B$  is the object $Z_\BM:=\pi_{Z!}(1_Z)$  in $\SH^G(B)$. For $v:=\V(E_\bullet)$, we have the twisted Borel-Moore motive $Z(v)_\BM:=
\pi_{Z!}(\Sigma^{v}(1_Z))$. If we need to denote the base-scheme $B$, we write these as $Z/B_\BM$ and $Z/B(v)_\BM$, respectively.
\end{definition}

 Let $p:Z\to W$ be a morphism in $\Sch^G/B$. For $p$ proper, we have the natural transformation ({\em proper pull-back})
 \begin{equation}\label{eqn:Projpush-forward}
 p^*:\pi_{W!}\to \pi_{Z!}\circ p^*
 \end{equation}
 defined as the composition  
\[
\pi_{W!} \xrightarrow{u_p}  \pi_{W!}p_*p^*\xrightarrow{(\eta^p_{!*})^{-1}}  \pi_{W!}p_!p^*\cong \pi_{Z!}\circ p^*.
\]
Applying $p^*$ to $\Sigma^{\V(E_\bullet)}(1_W)$ gives the morphism
\[
p^*:W(v)_\BM\to Z(p^*v)_\BM
\]
in $\SH^G(B)$. One checks easily that $(pq)^*=q^*p^*$ for composable proper morphisms $p$ and $q$. 

Let $f:Z\to W$ be a smooth morphism in $\Sch^G/B$.  We have the natural transformation  ({\em smooth push-forward})
\begin{equation}\label{eqn:Smoothpull-back}
f_*: \pi_{Z!}\Sigma^{T_f}f^*\to \pi_{W!}
\end{equation}
defined by  the composition
\[
\pi_{Z!} \Sigma^{T_f}f^*\cong \pi_{W!}f_! f^!\xrightarrow{e_f} \pi_{W!}.
\]
One checks that $f\mapsto f_*$ is functorial: For smooth morphisms $f:Z\to W$, $g:Y\to Z$, we have
\[
(fg)_*\circ\theta_{Y/Z/W}=f_*\circ [g_*\circ(\Sigma^{T_{fg}}\circ f^*)]
\]
where 
\[
\theta_{Y/Z/W}:(\pi_{Y!} \Sigma^{T_g}g^*)\circ(\Sigma^{T_f}\circ f^*)\to \pi_{Y!} \Sigma^{T_{fg}}\circ(fg)^*
\]
is the isomorphism induced by the exact sequence
\[
0\to T_g\to T_{fg}\to g^*T_f\to 0.
\]
Applying $f_*$ to $\Sigma^{\V(E_\bullet)}(1_W)$ gives the smooth push-forward map
\[
f_*:Z(f^*v+T_f)_\BM\to W(v)_\BM.
\]

Suppose a smooth morphism $f:Z\to W$ admits a section $s:W\to Z$.  We have the canonical isomorphism
\begin{equation}\label{eqn:a0}
f_!s_!\cong (fs)_!\tag{a}
\end{equation}
which induces the isomorphism on the adjoints
\begin{equation}\label{eqn:b0}
s^!f^!\cong (fs)^!\tag{b}
\end{equation}
Let $e_s:s_!s^!\to\id$, $e_f:f_!f^!\to \id$, $e_{fs}: (fs)_!(fs)^!\to \id$ be the co-units of adjunction. This gives us the commutative diagram
\[
\xymatrix{
 \id\ar@{}[r]|{\hbox{$\cong$}}\ar@{=}[d]&f_!s_!s^!f^!\ar[r]^-{f_!e_sf^!}\ar[d]_{(a)\circ(b)}^\wr&f_!f^!\ar[r]^{e_f}&\id\ar@{=}[d]\\
 \id \ar@{}[r]|{\hbox{$\cong$}}\ar@/_0.6cm/[rrr]_{\id}& (fs)_!(fs)^!\ar[rr]_{e_{fs}}&&\id,
 }
 \]
in other words, $f_!e_sf^!$  is a right inverse to $e_f$.\footnote{I am grateful to F. D\'eglise and D.C. Cisinski for this argument.}  Define the natural transformation
\begin{equation}\label{eqn:Section}
s_!:\pi_{W!}\to \pi_{Z!}\circ\Sigma^{T_f}\circ f^*
\end{equation}
as the composition
\[
\pi_{W!}\xrightarrow{\pi_{W!}\circ f_!e_sf^!}\pi_{W!}\circ f_!\circ f^!\cong
\pi_{W!}\circ f_!\circ \Sigma^{T_f}\circ f^*\cong\pi_{Z!}\circ\Sigma^{T_f}\circ f^*
\]
We call $s_!$ the {\em Gysin push-forward}.

\begin{lemma}\label{lem:SectionIsos} 1. Let $f:Z\to W$ be a smooth morphism in $\Sch^G/B$ with a section $s:W\to Z$. Then the composition
\[
\pi_{W!}\xrightarrow{s_!} \pi_{Z!}\circ\Sigma^{T_f}\circ f^*\xrightarrow{f_*}
\pi_{W!}
\]
is the identity.\\
2.  If $f:Z\to W$ is a vector bundle over $W$, then 
\[
f_*:\pi_{Z!}\circ\Sigma^{T_f}\circ f^*\to \pi_{W!}
\]
is an isomorphism, with inverse $s_!$.\\
3. If $f:Z\to W$ is an affine space bundle, then 
\[
f_*:\pi_{Z!}\circ\Sigma^{T_f}\circ f^*\to \pi_{W!}
\]
is an isomorphism.
\end{lemma}

\begin{proof} The fact that $e_f\circ(f_!e_sf^!)=\id$ implies (1).

For (2), it suffices by (1) to show that $f_*$ is an isomorphism. Since $Z\to W$ is a vector bundle, it follows by $\A^1$ homotopy invariance that $e_f:f_!f^!=f_\#f^*\to \id$ is an isomorphism; applying $\pi_{W!}\circ-$ yields (2).

As an affine space bundle is locally for the Zariski topology a vector bundle, (3) follows from (2) and Mayer-Vietoris.
\end{proof}
 
\begin{lemma}\label{lem:BaseChange}
Suppose we have a cartesian diagram in $\Sch^G/B$
\[
\xymatrix{
Z\ar[r]^-q\ar@{}@<-12pt>[r]_\Delta\ar[d]_g&Y\ar[d]^f\\
W\ar[r]_-p&X
}
\]
with $p$ proper and $f$ smooth. Then the diagram
\[
\xymatrixcolsep{3pt}
\xymatrix{
\pi_{Y!}\circ\Sigma^{T_f}\circ f^*\hskip20pt\ar[rrr]^-{q^*\circ(\Sigma^{T_f}\circ f^*)}\ar[d]_{f_*}
&&&\hskip20pt\pi_{Z!}\circ q^*\circ\Sigma^{T_f}\circ f^* \ar@{=}[r]&\pi_{Z!}\circ \Sigma^{T_g}\circ  g^*\circ p^*\ar[d]^{g_*\circ p^*}\\
\pi_{X!}\ar[rrrr]_{p^*}&&&&\pi_{W!}\circ p^*
}
\]
commutes. In other words,  proper pull-back commutes with smooth push-forward.
\end{lemma}

\begin{proof} In what follows we  simply write $\xrightarrow{\sim}$ for  isomorphisms that follow from functoriality, such as $\pi_{X!}f_!\cong\pi_{Y!}$ or that follow from the isomorphisms $\Sigma^{T_f}f^*\cong f^!$ or $\Sigma^{T_g}g^*\cong g^!$.

We fit a number of diagrams together.
\begin{equation}\label{eqn:a}\tag{a}
\xymatrixcolsep{30pt}
\xymatrix{
\pi_{Y!}\Sigma^{T_f}f^*\ar[r]^-{\pi_{Y!}\circ u_q}\ar[dd]_\wr&\pi_{Y!}q_*q^*\Sigma^{T_f}f^*\ar[d]^\wr\\
&\pi_{Y!}q_*q^*f^!\ar[d]^\wr\\
\pi_{X!}f_!f^!\ar[r]^{\pi_{X!}f_!u_qf^!}\ar[ddd]_{\pi_{X!}e_f}&\pi_{X!}f_!q_*q^*f^!\ar[d]^{\pi_{X!}Ex(\Delta_{!*})q^*f^!}\\
&\pi_{X!}p_*g_!q^*f^!\ar[d]^{\pi_{X!}p_*g_!Ex(\Delta^{*!})}_\wr\\
&\pi_{X!}p_*g_!g^!p^*\ar[d]^{\pi_{X!}p_*e_gp^*}\\
\pi_{X!}\ar[r]_{\pi_{X!}u_p}&\pi_{X!}p_*p^*
}
\end{equation}
\begin{equation}\label{eqn:b1}\tag{b1}
\xymatrixcolsep{60pt}
\xymatrix{
\pi_{Y!}q_*q^*\Sigma^{T_f}f^*\ar[d]^\wr&\pi_{Y!}q_!q^*\Sigma^{T_f}f^*\ar[l]^\sim_{\pi_{Y!}\eta^q_{!*}\Sigma^{T_f}f^*}
\ar[d]^\wr\\
\pi_{Y!}q_*q^*f^!\ar[d]^\wr&\pi_{Y!}q_!q^*f^! \ar[l]^\sim_{\pi_{Y!}\eta^q_{!*}f^!}\ar[d]^\wr\\
\pi_{X!}f_!q_*q^*f^!\ar[d]_{\pi_{X!}Ex(\Delta_{!*})}&\pi_{X!}p_!g_!q^*f^!\ar[ld]_\sim^{\pi_{X!}\eta^p_{!*}g_!q^*f^!}_\sim\\
\pi_{X!}p_*g_!q^*f^!&
}
\end{equation}
\begin{equation}\label{eqn:b2}\tag{b2}
\xymatrixcolsep{60pt}
\xymatrix{
\pi_{X!}p_*g_!q^*f^!\ar[d]^\wr_{\pi_{X!}p_*g_!Ex(\Delta^{*!})}&\pi_{X!}p_!g_!q^*f^!\ar[l]^\sim_{\pi_{X!}\eta^p_{!*}g_!q^*f^!}\ar[d]_\wr^{\pi_{X!}p_!g_!Ex(\Delta^{*!})}\\
\pi_{X!}p_*g_!g^!p^*\ar[d]_{\pi_{X!}p_*e_gp^*}&\pi_{X!}p_!g_!g^!p^*\ar[l]^\sim_{\pi_{X!}\eta^p_{!*}g_!g^!p^*}\ar[d]^{\pi_{X!}p_!e_gp^*}\\
\pi_{X!}p_*p^*&\pi_{X!}p_!p^*\ar[l]_\sim^{{\pi_{X!}\eta^p_{!*}p^*}}
}
\end{equation}

\begin{equation}\label{eqn:c}\tag{c}
\xymatrixcolsep{60pt}
\xymatrix{
\pi_{Y!}q_!q^*\Sigma^{T_f}f^*\ar[r]^\sim\ar[d]^\wr&\pi_{Z!}\Sigma^{T_g}g^*p^*\ar[d]^\wr\\
\pi_{Y!}q_!q^*f^!\ar[d]^\wr&\pi_{Z!}g^!p^*\ar[d]^\wr\\
\pi_{X!}p_!g_!q^*f^!\ar[r]_\sim^{\pi_{W!}g_!Ex(\Delta^{*!})}\ar[d]^\sim_{\pi_{W!}g_!Ex(\Delta^{*!})}
&\pi_{W!}g_!g^!p^*\ar[dl]^{\sim\circ\id_{g_!g^!p^*}}\ar[dd]^{\pi_{W!}e_gp^*}\\
\pi_{X!}p_!g_!g^!p^*\ar[d]_{\pi_{X!}p_!e_gp^*}\\
\pi_{X!}p_!p^*\ar[r]^\sim&\pi_{W!}p^*
}
\end{equation}

These fit together as
\[
\xymatrixcolsep{0pt}
\xymatrixrowsep{5pt}
\xymatrix{
\pi_{Y!}\circ\Sigma^{T_f}\circ f^*\ar[rrrr]^-{q^*\circ(\Sigma^{T_f}\circ f^*)}\ar[dddd]_{f_*}&&&&\pi_{Z!}\circ\Sigma^{T_{Z/W}}\circ g^*\circ p^*\ar[dddd]^{g_*\circ p^*}\\
&&(b1)&\\
&(a)&&(c)\\
&&(b2)&\\
\pi_{X!}\ar[rrrr]_{p^*}&&&&\pi_{W!}\circ p^*
}
\]
The  four diagrams (a), (b1), (b2) and (c) all commute; this follows from the commutativity of transformations acting on separate parts of a composition of functors, or the naturality of the unit and co-unit of an adjunction, or that fact that the exchange isomorphisms $Ex(\Delta^{*!})$ and $Ex(\Delta_{*!})$  are derived from the functoriality of composition for $(-)^*$ and $(-)_*$, combined with units and co-units of various adjunctions. For instance, the commutativity of the lower square in (a) is equivalent to the commutativity of the square
\[
\xymatrix{
f_\#f^*\ar[r]^-{u_q}\ar[ddd]_{e_f}&f_\#q_*q^*f^*\ar[d]^{Ex(\Delta_{\# *})}\\
&p_*g_\#q^*f^*\ar[d]^\wr\\
&p_*g_\#g^*p^*\ar[d]^{e_g}\\
\id\ar[r]_-{u_p}&p_*p^*
}
\]
We fill this in as follows
\[
\xymatrix{
f_\#f^*\ar[rrr]^-{u_q}\ar@{}@<-12pt>[rrr]_-{(i)}\ar[dr]^{u_p}\ar[dddd]_{e_f}\ar@{}@<15pt>[dddd]^{(v)}
&&&f_\#q_*q^*f^*\ar[dd]^{Ex(\Delta_{\# *})}\ar@{}@<-20pt>[dd]_{(ii)}
\ar[dl]_{u_p}\\
&p_*p^*f_\#f^*\ar[r]^-{u_q}\ar@{}@<-14pt>[r]_-{(iii)}&p_*p^*f_\#q_*q^*f^*\\
&p_*g_\#q^*f^*\ar[u]_\wr^{Ex(\Delta^*_\#)}\ar[r]_-{u_q}\ar[d]^\wr&p_*g_\#q^*q_*q^*f^*\ar[u]^\wr_{Ex(\Delta^*_\#)}\ar[r]_{e_q}
&p_*g_\#q^*f^*\ar[d]^\wr\ar@{}@<-96pt>[d]^{(iv)}\\
&p_*g_\#g^*p^*\ar[rr]^{\id}\ar[drr]_{e_g}\ar@{}@<-8pt>[rr]_(.7){(vi)}&&p_*g_\#g^*p^*\ar[d]^{e_g}\\
\id\ar[rrr]_-{u_p}&&&p_*p^*
}
\]
The commutativity of (i), (iii) and (vi) is obvious, that of (ii) is the definition of $Ex(\Delta_{\# *})$ and that of (iv) is the standard identity $(e_q\circ q^*)\circ(q^*\circ u_q)=\id$ for the unit and co-unit of an adjunction. The commutativity of (v) reduces to that of 
\[
\xymatrixcolsep{30pt}
\xymatrix{
f_\#f^*\ar[r]^-{u_p}\ar[d]_{e_f}&p_*p^*f_\#f^*\ar[d]_{e_f}
&p_*g_\#q^*f^*\ar[l]^\sim_{Ex(\Delta^*_\#)}\ar[d]^\wr\\
\id\ar[r]_-{u_p}&p_*p^*&p_*g_\#g^*p^*\ar[l]^{e_g}
}
\]
The commutativity of the left side is obvious and, using the definition of $Ex(\Delta^*_\#)$,  that of the right side reduces to the commutativity of
\[
\xymatrix{
g_\#q^*f^*\ar[d]_{u_f}&g_\#g_*p^*\ar[l]_\sim\ar[r]^-{e_g}&p^*\\
g_\#q^*f^*f_\#f^*\ar[r]_\sim &g_\#g^*p^*f_\#f^*\ar[r]_-{e_g}&p^*f_\#f^*\ar[u]_{e_f}
}
\]
Filling this in as
\[
\xymatrix{
&g_\#g_*p^*\ar[dl]_\sim\ar[r]^-{e_g}&p^*\\
g_\#q^*f^*\ar[d]_{u_f}\ar@{=}[r]&g_\#q^*f^*\ar[u]_\wr\\
g_\#q^*f^*f_\#f^*\ar[r]_\sim\ar[ru]_{e_f} 
&g_\#g^*p^*f_\#f^*\ar[r]_-{e_g}&p^*f_\#f^*\ar[uu]_{e_f}
}
\]
we see that the commutativity follows from the identity $(e_f\circ f^*)\circ(f^*\circ u_f)=\id$.

The commutativity of the remaining diagrams is much easier to verify and we leave the details to the reader.
\end{proof}

\begin{remark} \label{rem:basechange} Proper pull-back, smooth push-forward and  Gysin push-forward  are all compatible with base-change in the following sense. Fix a morphism $g:B'\to B$. 

 Let $p:Z\to W$ in $\Sch^G/B$ be proper. Using the base-change isomorphisms $Ex(\Delta^*_!)$, we see that the proper pull-back morphism $p^*:\pi_{W!}\to \pi_{Z!}\circ p^*$ is natural with respect to pull-back  by  $g:B'\to B$. In detail, we have the cartersian square
 \[
 \xymatrix{
Z_{B'}\ar[r]^{g_Z}\ar[d]_{p_{B'}}&Z\ar[d]^p\\
W_{B'}\ar[r]^{g_W}&W
}
\]
with $Z_{B'}=Z\times_BB'$, $W_{B'}=W\times_BB'$, both considered as objects in $\Sch^G/B'$. The base-change isomorphisms $Ex(\Delta^*_!)$ gives the commutative diagram 
 \[
\xymatrix{
\pi_{W_{B'}!}\circ g_W^*\ar[r]_-\sim\ar[d]_{p_{B'}^*}&g^*\circ \pi_{W!}\ar[d]^{p^*}\\
\pi_{Z_{B'}!}\circ p_{B'}^*\circ g_W^*\ar[r]_-\sim&g^*\circ\pi_{Z!}
}
\]

For $f:Z\to W$ smooth in $\Sch^G/B$, inducing $f_{B'}:Z_{B'}\to W_{B'}$,  the base-change isomorphisms $Ex(\Delta^*_!)$ and $Ex(\Delta^{*!})$, and the naturality of $\Sigma^{-}$ gives us the commutative diagram
\[
\xymatrix{
\pi_{Z_{B'}!}\circ \Sigma^{T_{f_{B'}}}\circ f_{B'}^*\circ g_W^*\ar[r]_-\sim\ar[d]_{f_{B*}}&
g^*\circ \pi_{Z!}\circ \Sigma^{T_{f}}\circ f^*\ar[d]^{f_*}\\
\pi_{W_{B'}!}\circ g_W^*\ar[r]_-\sim&g^*\circ \pi_{W!}
}
\]
If $s:W\to Z$ is a section to a smooth $f:Z\to W$,  we have the induced section $s_{B'}:W_{B'}\to Z_{B'}$ to $f_{B'}$, and  the base-change isomorphisms $Ex(\Delta^*_!)$ and $Ex(\Delta^{*!})$, and the naturality of $\Sigma^{-}$ give us the commutative diagram 
\[
\xymatrix{
\pi_{Z_{B'}!}\circ \Sigma^{T_{f_{B'}}}\circ f^{\prime*}\circ g_W^*\ar[r]_-\sim&
g^*\circ \pi_{Z!}\circ \Sigma^{T_{f}}\circ f^*\\
\pi_{W_{B'}!}\circ g_W^*\ar[r]_-\sim\ar[u]^{s_{B!}}&g^*\circ\pi_{W!}\ar[u]^{s_!}
}
\]
These are all diagrams of functors from $\SH^G(W)$ to $\SH^G(B')$, and all the arrows marked with a ``$\sim$'' are isomorphisms. 
 \end{remark}

In the next definition, we adapt the notation for a bivariant theory introduced in \cite{DJK}. 
\begin{definition}\label{def:BMHomology} For $\sE\in \SH^G(B)$, $\pi_Z:Z\to B$ in $\Sch^G/B$, $E_\bullet\in \sV^G(Z)$ and $v=\V(E_\bullet)$, define the {\em twisted Borel-Moore homology} with values in $\sE$ as
\[
\sE^\BM_{a,b}(Z, v):=\Hom_{\SH^G(B)}(\Sigma^{a,b}Z(v)_\BM, \sE)
\]
To simplify the notation, we write $\sE^\BM(Z, v)$ for $\sE^\BM_{0,0}(Z, v)$.

For an object $\sF\in \SH^G(B)$, we have the {\em  $\sE$-cohomology}
\[
\sE^{a,b}(\sF):=\Hom_{\SH^G(B)}(\sF, \Sigma^{a,b}\sE);
\]
for $\sX\in \sH^G_\bullet(B)$ define $\sE^{a,b}(\sX):=\sE^{a,b}(\Sigma^\infty_T\sX)$ and for $\sY\in 
\sH^G(B)$, define $\sE^{a,b}(\sY):=\sE^{a,b}(\sY_+)$. Finally, for $X\in \Sm^G/B$, $E_\bullet\in \sV^G(X)$ and $v=\V(E_\bullet)$ define the twisted $\sE$-cohomology
\[
\sE^{a,b}(X,v):=\sE^{a,b}(\pi_{X\#}(\Sigma^{-v}(1_X))=\Hom_{\SH^G(B)}(\Sigma^\infty_TX_+, \Sigma^{a,b}\Sigma^v\sE)
\]
\end{definition}

\begin{remark}\label{rem:GysinSmoothCartesian} Let
 \[
 \xymatrix{
 Z'\ar[r]^q\ar[d]_g&Z\ar[d]^f\\
 W'\ar[r]^p&W
 }
 \] 
 be a cartesian diagram in $\Sch^G/B$ with $p, q$ smooth and with $f:Z\to W$ a vector bundle. Let $s:W\to Z$ be a section and let $t:W'\to Z'$ be the induced section. Then the diagram
 \[
 \xymatrix{
(Z, f^*v+T_f)_\BM &&(W, v)_\BM\ar[ll]_-{s_!}\\
(Z', q^*(f^*v+ T_f)+T_q)_\BM \ar[u]^{q_*}&
(Z', g^*(p^*v+ T_p)+T_g)_\BM \ar@{=}[l]&(W', p^*v+T_p)_\BM\ar[u]^{p_*}\ar[l]^-{t_!}
  }
\]
commutes.  To see this, we may  replace the maps $s_!$, $t_!$ with their respective inverses $f_*$, $g_*$ (Lemma~\ref{lem:SectionIsos}), and then the commutativity follows from the functoriality of smooth push-forward: $p_*g_*=f_*q_*$.

In fact, the Gysin push-forward commutes with smooth push-forward in cartesian squares, without assuming the smooth maps are vector bundles, but as we do not need this result, we omit the proof.
\end{remark}

\begin{remark} The usual operations on Borel-Moore homology: proper push-forward, smooth pull-back, intersection with a section, all follow by applying proper pull-back $p^*$, smooth push-forward $f_*$ or  Gysin push-forward $s_!$ to morphisms $\Sigma^{**}(-)/B^\BM(-)\to \sE$. Explicitly, a proper map $p:Z\to W$ in $\Sch^G/B$ induces the functorial proper push-forward
\[
p_*: \sE^\BM_{a,b}(Z, p^*v)\to \sE^\BM_{a,b}(W, v)
\]
defined by $p_*:=(p^*)^*$,  a smooth map $f:Z\to W$   in $\Sch^G/B$ induces the functorial smooth pull-back
\[
f^*:\sE^\BM_{a,b}(W, v)\to\sE^\BM_{a,b}(Z, f^*v+T_f)
\]
defined by $f^*:=(f_*)^*$, and for $f:Z\to W$ a smooth morphism in $\Sch^G/B$ with a section $s:W\to Z$, we have the Gysin map
\[
s^!:\sE^\BM_{a,b}(Z, f^*v+T_f)\to \sE^\BM_{a,b}(W, v)
\]
defined by $s^!:=(s_!)^*$, and satisfying $s^!\circ f^*=\id$.  We hope that the context will enable the reader  to distinguish the {\em maps} $p_*$, $f^*$ and $s^!$ from the {\em functors} $p_*$, $f^*$ and $s^!$.

 With this translation, Lemma~\ref{lem:BaseChange} says that on twisted Borel-Moore homology,  proper push-forward commutes with smooth pull-back in a cartesian square, and Remark~\ref{rem:GysinSmoothCartesian} says that the Gysin map commutes with smooth pull-back  in a cartesian square of vector bundles.

For $\pi_X:X\to B$ in $\Sm^G/B$, the purity isomorphism $\pi_{X!}\cong \pi_{X\#}\circ \Sigma^{-T_X}$ gives the isomorphism
\[
\sE^\BM_{a,b}(X, v)\cong  \sE^{-a,-b}(X, T_X-v).
\]

Finally, for $f:Z\to X$ a morphism in $\Sch^G/B$ with $X\in \Sm^G/B$, and $\sE\in \SH^G(B)$ a commutive ring spectrum (i.e. commutative monoid object in $\SH^G(B)$), we have the cap product
\[
-\cap f^*(-): \sE^\BM_{a,b}(Z, v)\times \sE^{p,q}(W, w)\to 
\sE^\BM_{a-p,b-q}(Z, v-f^*w)
\]
defined via the monoidal structure
\begin{multline*}
\Hom_{\SH^G(B)}(\Sigma^{a,b}Z(v)_\BM, \sE)\times \Hom_{\SH^G(B)}(\Sigma^\infty_TW_+, \Sigma^{p,q}\Sigma^w\sE)\\\to 
\Hom_{\SH^G(B)}(\Sigma^{a-p,b-q}Z(v)_\BM\wedge  \Sigma^{-w}\Sigma^\infty_TW_+,\sE\wedge\sE),
\end{multline*}
followed by the multiplication map $\sE\wedge\sE\to \sE$,
the purity isomorphism 
\[
Z(v)_\BM\wedge  \Sigma^{-w} \Sigma^\infty_TW_+\cong Z\times_BW(p_1^*v-p_2^*w+p_2^*T_W)_\BM
\]
and the Gysin map for the graph morphism $\gamma_f:Z\to Z\times_BW$
\[
\gamma_f^!:\sE^\BM_{a-p,b-q}(Z\times_BW, p_1^*v-p_2^*w+p_2^*T_W)\to
\sE^\BM_{a-p,b-q}(Z, v-f^*w).
\]

\end{remark}

\section{The intrinsic stable normal cone}\label{sec:NormalThom}
Take $Z\in \Sch^G/B$. Since $Z$ is $G$-quasi-projective over $B$,  $Z$ admits a closed immersion $i:Z\to M$ in $\Sch^G/B$ with $M\in \Sm^G/B$. As in \cite{BF}, we have the {\em normal cone} $\mathfrak{C}_{i}$:
\[
\mathfrak{C}_{i}:=\Spec_{\sO_X}(\oplus_{n\ge0}\sI_Z^n/\sI_Z^{n+1}),
\]
where $\sI_Z$ is the ideal sheaf of $Z\subset M$.  Let $p_i:\mathfrak{C}_{i}\to Z$ be the projection and $\sigma_i:\mathfrak{C}_i\to M$ the composition $i\circ p_i$. As before, we denote the structure morphism for $Y\in \Sch^G/B$ by $\pi_Y:Y\to B$. For $f:Y\to Z$ a smooth morphism in $\Sch^G/B$, we often denote the relative tangent bundle $T_f$ by $T_{Y/Z}$ and in case $f=\pi_Y:Y\to B$ is the (smooth) structure morphism we may write $T_Y$ instead of $T_{Y/B}$

\begin{lemma} \label{lem:CanonIso} Suppose we have closed immersions $i:Z\to M$, $i':Z\to M'$ in $\Sch^G_B$ with $M, M'\in \Sm^G/B$. Then there is a canonical isomorphism 
\[
\psi_{i,i'}: \mfC_{i'}(\sigma_{i'}^*T_{M'})_\BM\xrightarrow{\sim} \mfC_{i}(\sigma_i^* T_{M})_\BM.
\]
If we have a third closed immersion $i'':Z\to M''$ then $\psi_{i,i'}\circ \psi_{i',i''}=\psi_{i,i''}$.
\end{lemma}

\begin{proof} Suppose we have defined the required isomorphism $\psi_g:=\psi_{i,i'}$ for each commutative diagram in $\Sch^G_B$
\begin{equation}\label{eqn:SmoothLift}
\xymatrix{
Z\ar[d]_{i'}\ar[dr]^i\\
M'\ar[r]_g&M
}
\end{equation}
with $g$ a smooth morphism in $\Sm^G/B$ and $i, i'$ closed immersions, satisfying $\psi_{gg'}=\psi_g\circ\psi_{g'}$ for each commutative diagram
\[
\xymatrix{
&Z\ar[dl]_{i''}\ar[d]^{i'}\ar[dr]^i\\
M''\ar[r]_{g'}&M'\ar[r]_g
&M
}
\]
For an arbitrary pair $i:Z\to M$, $i':Z\to M'$, we have the closed immersion $(i,i'):Z\to M\times_BM'$ and set $\psi_{i,i'}:=\psi_{p_1}\circ\psi_{p_2}^{-1}$, which solves our problem.

We consider a commutative diagram \eqref{eqn:SmoothLift}.  The smooth morphism $g$ induces a smooth morphism $\mfC(g):\mfC_{i'}\to \mfC_{i}$, giving  the commutative diagram
\[
\xymatrix{
\mfC_{i'}\ar[r]^{\mfC(g)}\ar[d]_{\sigma_{i'}}&\mfC_{i}
\ar[d]^{\sigma_i}\\
M'\ar[r]_g&M
}
\]
The projection $\mfC(g)$ makes $\mfC_{i'}$ into a torsor over $\mfC_{i}$ for the vector bundle $p_i^*i^{\prime*}T_{M'/M}$, giving a canonical identification
\[
T_{\mfC_{i'}/\mfC_{i}}\cong \sigma_{i'}^*T_{M'/M}.
\]

The exact sequence
\[
0\to T_{M'/M}\to T_{M'}\to g^*T_{M}\to 0
\]
gives the canonical isomorphism 
\[
 \Sigma^{\sigma_{i'}^*T_{M'}}  \xrightarrow{\theta_g}\Sigma^{\sigma_{i'}^*T_{M'/M}}
\circ \Sigma^{\mfC(g)^*\sigma_i^*T_{M}}.
\]

By Lemma~\ref{lem:SectionIsos}, the smooth push-forward 
\[
\mfC(g)_*:\pi_{\mfC_{i'}!}\circ\Sigma^{\sigma_{i'}^*T_{M'/M}}\circ \mfC(g)^*\to
\pi_{\mfC_{i}!}
\]
is an isomorphism. 

Composing $\mfC(g)_*$ with   the isomorphism $\Theta_g$, defined as the composition
\begin{multline*}
\pi_{\mfC_{i'}!}\circ\Sigma^{\sigma_{i'}^*T_{M'}}\circ \mfC(g)^*\xymatrix{\ar[r]^{\theta_g}_\sim&}
\pi_{\mfC_{i'}!}\circ\Sigma^{\sigma_{i'}^*T_{M'/M}}
\circ \Sigma^{\mfC(g)^*\sigma_i^*T_{M}}\circ \mfC(g)^*\\\cong
\pi_{\mfC_{i'}!}\circ\Sigma^{\sigma_{i'}^*T_{M'/M}}\circ \mfC(g)^*\circ \Sigma^{\sigma_i^*T_{M}}
\end{multline*}
gives us the isomorphism
\[
 \mfC(g)_*\circ\Theta_g:\pi_{\mfC_{i'}!}\circ\Sigma^{\sigma_{i'}^*T_{M'/B}}\circ \mfC(g)^*\to 
\pi_{\mfC_{i}!}\circ \Sigma^{\sigma_i^*T_{M}}.
\]
Evaluating at $1_{\mfC_{i}}$ gives the isomorphism
\[
\psi_g: \mfC_{i'}(\sigma_{i'}^*T_{M'})_\BM\to
\mfC_{i}(\sigma_i^*T_{M})_\BM.
\]

Suppose we have another smooth morphism $g':M''\to M'$ and a closed immersion $i'':Z\to M''$ with $g'\circ i''=i'$. Let $\Xi_{g,g'}$ be the isomorphism
\begin{multline*}
\pi_{\mfC_{i''}!}\circ\Sigma^{\sigma_{i''}^*T_{M''/M}}\circ\mfC(gg')^*\circ\Sigma^{\sigma_i^*T_{M}}\\
\xrightarrow{\Xi_{g,g'}} 
\pi_{\mfC_{i''}!}\circ\Sigma^{\sigma_{i''}^*T_{M''/M'}}\circ\mfC(g')^*\circ\Sigma^{\sigma_{i'}^*T_{M'/M}}\circ\mfC(g)^*\circ \Sigma^{\sigma_i^*T_{M}}
\end{multline*}
defined as the composition
\begin{multline*}
\pi_{\mfC_{i''}!}\circ\Sigma^{\sigma_{i''}^*T_{M''/M}}\circ\mfC(gg')^*\circ\Sigma^{\sigma_i^*T_{M}}
\\\xrightarrow{\theta_{g'}}
\pi_{\mfC_{i''}!}\circ\Sigma^{\sigma_{i''}^*T_{M''/M'}}\circ\Sigma^{\mfC(g')^*\sigma_{i'}^*T_{M'/M}}\circ\mfC(g')^*\circ \mfC(g)^*\circ\Sigma^{\sigma_i^*T_{M}}\\\cong 
\pi_{\mfC_{i''}!}\circ\Sigma^{\sigma_{i''}^*T_{M''/M'}}\circ\mfC(g')^*\circ\Sigma^{\sigma_{i'}^*T_{M'/M}}\circ\mfC(g)^*\circ \Sigma^{\sigma_i^*T_{M}}
\end{multline*}
The functoriality of smooth push-forward gives the identity
\begin{equation}\label{eqn:Fun}
\mfC(gg')_*=\mfC(g)_*\circ \mfC(g')_*\circ\Xi_{g,g'}.
\end{equation}
as maps from $\pi_{\mfC_{i''}!}\circ \Sigma^{\sigma_{i''}^*T_{M''}}\circ \mfC(gg')^*\circ\Sigma^{\sigma_{i}^*T_{M}}$ to 
$\pi_{\mfC_i!}\circ\Sigma^{\sigma_{i}^*T_{M}}$.

We then have
\begin{align*}
\psi_g\circ \psi_{g'}&=\mfC(g)_*\circ \Theta_g\circ \mfC(g')_*\circ \Theta_{g'}(1_{\mfC_i})\\
&=\mfC(g)_*\circ\mfC(g')_*\circ \Theta'_g\circ \Theta_{g'}(1_{\mfC_i})\\
&=\mfC(g)_*\circ\mfC(g')_*\circ\Xi_{g,g'}\circ \Theta_{gg'}(1_{\mfC_i})\\
&=\mfC(gg')_*\circ \Theta_{gg'}(1_{\mfC_i})\\
&=\psi_{gg'},
\end{align*}
where  $\Theta'_g$  is the isomorphism
\begin{multline*}
\pi_{\mfC_{i''}}\circ\Sigma^{\sigma_{i''}^*T_{M''/M'}}\circ\mfC(g')^*\circ\Sigma^{\sigma_{i'}^*T_{M'}}\circ \mfC(g)^*\\\to
\pi_{\mfC_{i''}}\circ\Sigma^{\sigma_{i''}^*T_{M''/M'}}\circ\mfC(g')^*\circ\Sigma^{\sigma_{i'}^*T_{M'/M}}\circ \mfC(g)^*\circ \Sigma^{\sigma_{i'}^*T_{M}}
\end{multline*}
constructed similarly to $\Theta_g$.  All the above identities are easy consequences of the naturality of the functor $\Sigma^{-}:D^\perf_{G, iso}(-)\to \Aut\SH^G(-)$, and \eqref{eqn:Fun}.
\end{proof}

Relying on Lemma~\ref{lem:CanonIso}, we can prove the main result of this section: the construction of an analog of the Behrend-Fantechi intrinsic normal cone in the setting of the stable motivic homotopy category.

\begin{theorem}\label{thm:IntStableNormalCone} For each $Z\in \Sch^G/B$, there is an object $\Cst_Z\in \SH^G(B)$ satisfying the following:\\[5pt]
1. Let $i:Z\to M$ be a closed immersion in $\Sch^G/B$ with $M$ smooth over $B$. Then there is a  canonical isomorphism
\[
\alpha_i:\Cst_Z\xrightarrow{\sim} \mfC_{i}(\sigma_i^*T_{M})_\BM
\]
2. For $i':Z\to M$ another closed immersion in $\Sch^G/B$,   have 
\[
\alpha_i=\psi_{i,i'}\circ\alpha_{i'},
\]
where $\psi_{i,i'}:\mfC_{i'}(\sigma_{i'}^*T_{M'})_\BM\xrightarrow{\sim}
\mfC_{i}(\sigma_i^*T_{M})_\BM$ is the isomorphism defined in 
Lemma~\ref{lem:CanonIso}.
\end{theorem}

\begin{proof} Choose a closed immersion $i_0:Z\to M_0$ in $\Sch^G/B$ with $M_0\in \Sm^G/B$, which exists by the definition of the category $\Sch^G/B$, and define
\[
\Cst_Z:=\mfC_{i_0}(\sigma_{i_0}^*T_{M_0})_\BM. 
\]
For   $i:Z\to M$ a closed immersion $\Sch^G/B$ with $M\in \Sm^G/B$, Lemma~\ref{lem:CanonIso} gives us   the   isomorphism
\[
\alpha_i:=\psi_{i,i_0}:\Cst_Z\to \mfC_{i}(\sigma_i^*T_{M})_\BM
\]
and shows that if $i':Z\to M'$ is another closed immersion,  the diagram
\[
\xymatrix{
\Cst_Z\ar[r]^-{\alpha_{i'}}\ar[dr]_-{\alpha_i}&
\mfC_{i'}(\sigma_{i'}^*T_{M'})_\BM\ar[d]^{\psi_{i,i'}}\\
&\mfC_{i}(\sigma_i^*T_{M})_\BM
}
\]
commutes. 
\end{proof}

\begin{definition} \label{Def:IntStableNormalCone} For  $Z$  in $\Sch^G/B$, we call $\Cst_Z\in \SH^G(B)$ the {\em intrinsic stable normal cone} of $Z$. 
\end{definition}

\begin{ex} Suppose $\pi_Z:Z\to B$ is smooth over $B$. Then we may take the identity for $i:Z\to M$, giving $\mfC_{i}=Z$ and $\Cst_Z=\pi_{Z!}\circ \Sigma^{T_{Z}}(1_Z)=\pi_{Z\#}(1_Z)=\Sigma^\infty_TZ_+$ in $\SH^G(B)$. 
\end{ex}

\begin{remark} Here is an intuitive justification for our definition of $\Cst_Z$. Let $V\to Z$ be a rank $r$ vector bundle on a $B$-scheme $Z$ and let $p:C\to Z$ be a morphism in $\Sch^G/B$. We consider $V$ as a group-scheme over $Z$ via fiberwise vector addition.

Suppose we have a  (smooth) morphism $q:C\to \bar{C}$ over $Z$ that makes $C$ a principal $V$-bundle 
over $\bar{C}$; let $\bar{p}:\bar{C}\to Z$ be the structure morphism and let $\pi_C:C\to B$, $\pi_{\bar{C}}:\bar{C}\to B$ be the structure morphisms to $B$. We have a canonical isomorphism $T_{C/\bar{C}}\cong p^*V$, the purity isomorphism $q_!\circ\Sigma^{p^*V}\cong q_\#$ and the homotopy invariance isomorphism $1_{\bar{C}}\cong q_\#(1_C)$.  This gives the isomorphism in $\SH^G(B)$
\begin{multline*}
\bar{C}_\BM:=\pi_{\bar{C}!}(1_{\bar{C}})\cong \pi_{\bar{C}!}(q_\#(1_C))\cong
\pi_{\bar{C}!}\circ q_!\circ\Sigma^{p^*V}(1_C)\cong \pi_{C!}(\Sigma^{p^*V}(1_C))=C(p^*V)_\BM .
\end{multline*}

Now suppose  we just have a action of $V$ on $C$ over $Z$, giving us the quotient stack $[V\backslash C]$ with projection $\pi:C\to [V\backslash C]$. In the category of Artin stacks, $\pi$ is a principal $V$-bundle, so we could reasonably {\em define} the Borel-Moore motive of $[V\backslash C]$ by 
\[
[V\backslash C]_\BM:=C(p^*V)_\BM.
\]

The cone $p:\mfC_i\to Z$ for a closed immersion $i:Z\to M$ has the canonical action by $i^*T_M$, giving the  Behrend-Fantechi intrinsic normal cone $\mfC_Z:=[i^*T_M\backslash \mfC_i]$.  We should thus define the Borel-Moore motive of  $\mfC_Z$ by
\[
(\mfC_Z)_\BM=[i^*T_M\backslash \mfC_i]_\BM:=\mfC_i(\sigma_i^*T_M)_\BM,
\]
which is exactly our definition of  $\Cst_Z$.
\end{remark}

\begin{lemma}\label{lem:HomotopyInv} Suppose we have a closed immersion $i:Z\to M$ with $M\in\Sm^G/B$ and an affine space bundle $q:V\to M$, giving the cartesian diagram
\[
\xymatrix{
Z'\ar[r]^{i'}\ar[d]_{q_Z}&V\ar[d]^q\\
Z\ar[r]_i&M
}
\]
Let $\mfC(q):\mfC_{i'}\to \mfC_{i}$ be the morphism induced by $(q_Z, q)$ and let 
\begin{equation}\label{eqn:StableConeIso}
\Cst(q):\Cst_{Z'}\to \Cst_Z
\end{equation}
be defined as the composition
\begin{align*}
\Cst_{Z'}&\xymatrix{\ar[r]^{\alpha_{i'}}&}\pi_{\mfC_{i'}!}\circ \Sigma^{\sigma_{i'}^*T_{V}}(1_{\mfC_{i'}})\\
&\xymatrix{\ar[r]^\theta_\sim&} \pi_{\mfC_{i'}!}\circ \Sigma^{\sigma_{i'}^*q^*V}\circ \Sigma^{\sigma_{i'}^*q^*T_{M}}\circ \mfC(q)^*(1_{\mfC_{i}})\\
&\xymatrix{\ar[r]_\sim&}  \pi_{\mfC_{i'}!}\circ \Sigma^{\sigma_{i'}^*q^*V}\circ \Sigma^{\mfC(q)^*\sigma_{i}^*T_{M}}\circ \mfC(q)^*(1_{\mfC_{i}})\\
&\xymatrix{\ar[r]_\sim&}  \pi_{\mfC_{i'}!}\circ \Sigma^{\sigma_{i'}^*q^*V}\circ \mfC(q)^*\circ\Sigma^{\sigma_{i}^*T_{M}}(1_{\mfC_{i}})\\
&\hskip2pt\xrightarrow{\mfC(q)_*}
 \pi_{\mfC_{i}!}\circ\Sigma^{\sigma_{i}^*T_{M}}(1_{\mfC_{i}})\\\
&\xymatrix{\ar[r]^{\alpha_i^{-1}}&}\Cst_Z,
\end{align*}
where the  isomorphism $\theta$ is induced by the exact sequence
\[
0\to q^*V\to T_{V}\to q^*T_{M}\to 0.
\]
Then $\Cst(q)$ is an isomorphism. Moreover, if we have an extension of our diagram to a cartesian diagram in $\Sch^G/B$
\[
\xymatrix{
Z''\ar[r]^{i''}\ar[d]_{q'_Z}&W\ar[d]^{q'}\\
Z'\ar[r]^{i'}\ar[d]_{q_Z}&V\ar[d]^q\\
Z\ar[r]_i&M
}
\]
with $q':W\to V$ an affine space bundle, then
\[
\Cst(q\circ q')=\Cst(q)\circ \Cst(q').
\]
\end{lemma}

\begin{proof} We use the closed immersion $i':Z'\to V$ to compute $\Cst_{Z'}$. The  morphism
\[
\mfC(q):\mfC_{i'}\to \mfC_{i},
\]
identifies $\mfC_{i'}$ with $\mfC_{i}\times_MV$, and hence by Lemma~\ref{lem:SectionIsos}, the map $\mfC(q)_*$ is an isomorphism, which implies that $\Cst(q)$ is an isomorphism as well.

The  functoriality $\Cst(q\circ q')=\Cst(q)\circ \Cst(q')$ follows the functoriality of smooth push-forward, the naturality of the isomorphisms $\theta_{-}$ and  naturality of $\Sigma^{-}:\sV^G(-)\to \Aut(\SH^G(-))$, as in the proof of Lemma~\ref{lem:CanonIso}.
\end{proof}

\section{The fundamental class}\label{sec:FundClass} The next step in the construction is to define a fundamental class in co-homotopy of our intrinsic stable normal cone: $[\Cst_Z]\in \mS_B^{0,0}(\Cst_Z)$. We do this by the method of specialization to the normal cone, suitably interpreted.

Choose as before a closed immersion $i:Z\to M$ in $\Sch^G/B$ with $M\in \Sm^G/B$. Let  
\[
\tilde{\pi}:\widetilde{M\times\A^1}\to M\times\A^1 
\]
be the blow up of $M\times\A^1$ along $Z\times0$, that is
\[
\widetilde{M\times\A^1}=\Proj_{M\times\A^1}\oplus_{n\ge0}\sI_{Z\times0}^n.
\]
Writing $M\times\A^1=\Spec\sO_M[t]$, the element $t\in \sI_{Z\times0}$, considered as an element of  $\oplus_{n\ge0}\sI_{Z\times0}^n$ of degree one, gives a $G$-invariant section of $\sO(1)$, which we denote by $T$. We define 
\[
Def(i):=\widetilde{M\times\A^1}\setminus (T=0)\in \Sch^G/B,
\]
so 
\[
Def(i)=\Spec_{M\times\A^1}(\oplus_{n\ge0}\sI_{Z\times0}^n[T^{-1}])_0,
\]
where the subscript 0 denotes the subsheaf of homogeneous sections of degree 0. The projection $p:Def(i)\to M\times\A^1$ is flat, $p^{-1}(M\times(\A^1\setminus\{0\})$ is isomorphic via $p$ to $M\times(\A^1\setminus\{0\})$ and $p^{-1}(M\times 0)=\mfC_{i}$. Thus $\mfC_{i}$ is a principal effective Cartier divisor on $Def(i)$, with ideal sheaf $(t)\sO_{Def(i)}$. Let $i_{\mfC}:\mfC_{i}\to Def(i)$ be the inclusion, and let $j:M\times(\A^1\setminus\{0\})\to Def(i)$ be the open complement. 

The localization triangle \eqref{eqn:LocTria}, twisted by $\Sigma^{p^*p_1^*T_{M/B}}$ and pushed forward by $\pi_{Def(i)!}$, gives us the distinguished triangle in $\SH^G(B)$
\[
\pi_{M\times(\A^1\setminus\{0\})!}\circ \Sigma^{p_1^*T_{M}}(1_{M\times(\A^1\setminus\{0\})})\to 
\pi_{Def(i)!}\circ \Sigma^{p^*p_1^*T_{M}}(1_{Def(i)})\to \mfC_{i}(\sigma_i^*T_{M})_\BM\xrightarrow{+1}.
\]
The isomorphism 
\[
T_{M\times (\A^1\setminus\{0\})}\cong p_1^*T_{M}\oplus p_2^*T_{ \A^1\setminus\{0\}}
\]
and the canonical isomorphism $T_{ \A^1\setminus\{0\}}\cong \sO_{ \A^1\setminus\{0\}}$ gives the isomorphisms
\begin{align*}
\pi_{M\times(\A^1\setminus\{0\})!}\circ \Sigma^{p_1^*T_{M}}(1_{M\times(\A^1\setminus\{0\})})&\cong \pi_{M\#}(1_M)\wedge_B\pi_{\A^1\setminus \{0\}\#}(\Sigma_T^{-1}1_{\A^1\setminus \{0\}})\\
&\cong \Sigma_T^{-1}M\times(\A^1\setminus\{0\})_+\\
&\cong \Sigma_{S^1}^{-1}\Sigma_{\G_m}^{-1}M\times(\A^1\setminus\{0\})_+
\end{align*}

Our distinguished triangle thus  gives us the map in $\SH^G(B)$
\[
\del:\mfC_i(\sigma_i^*T_M)_\BM\to \Sigma_{\G_m}^{-1} M\times(\A^1\setminus\{0\})_+.
\]
Let  
\[
\bar{p}^M_2:M\times(\A^1\setminus\{0\})_+\to \G_m
\]
be the projection $p_2:M\times
(\A^1\setminus\{0\})_+\to (\A^1\setminus\{0\})_+$ followed by the quotient map $(\A^1\setminus\{0\})_+\to (\A^1\setminus\{0\},\{1\})=\G_m$. This gives us the map
\[
\Sigma_{\G_m}^{-1}\bar{p}^M_2\circ \del:\mfC_i(\sigma_i^*T_M)_\BM\to \mS_B
\]

\begin{definition}\label{def:FundClass0} The fundamental class $[\mfC_i]\in \mS^\BM_B(\mfC_i, \sigma_i^*T_M)$ is the element represented by $\Sigma_{\G_m}^{-1}\bar{p}^M_2\circ \del$.

If $\sE$ is a commutative monoid in $\SH^G(B)$ with unit $\epsilon_\sE:\mS_B\to \sE$, we define the fundamental class $[\mfC_i]_\sE\in \sE^\BM(\mfC_i, \sigma_i^*T_M)$ by composing $[\mfC_i]$ with $\epsilon_\sE$. \end{definition}

We have  the canonical isomorphism $\alpha_i:\Cst_Z\to  \mfC_{i}(\sigma_i^*T_{M})_\BM$,
giving the map in $\SH^G(B)$
\begin{equation}\label{eqn:Boundary}
[\sC_i]\circ\alpha_i:\Cst_Z\to  \mS_B
\end{equation}

\begin{lemma}\label{lem:FundClassInd} The map 
\[
[\sC_i]\circ\alpha_i:\Cst_Z\to   \mS_B
\]
is independent of the choice of closed immersion $i:Z\to M$.
\end{lemma}

\begin{proof}  We reduce as in the proof of Lemma~\ref{lem:CanonIso} to the case in which we have a commutative diagram
\[
\xymatrix{
Z\ar[r]^{i'}\ar[rd]_i&M'\ar[d]^g\\
&M}
\]
with $g$ smooth. The map $g$ induces a smooth morphism $Def(g):Def(i')\to Def(i)$,  giving us the commutative diagram
\[
\xymatrix{
Def(i')\ar[r]^{p_{M'}}\ar[d]_{Def(g)}&M'\times\A^1\ar[d]^{g\times\id}\\
Def(i)\ar[r]_{p_M}&M\times\A^1}
\]
The restriction of $Def(g)$ to $\mfC_{i'}$ is the map $\mfC(g):\mfC_{i'}\to \mfC_{i}$  induced by $g$.

This gives us the map of distinguished triangles (we suppress the isomorphisms on the suspension operations induced by various exact sequences)
\[
\xymatrix{
\pi_{M'\times(\A^1\setminus\{0\})!}\circ \Sigma^{p_1^*T_{M'}}(1_{M'\times(\A^1\setminus\{0\})})
\hskip-50pt
\ar[dd]\ar[rd]^{\hskip30pt(g\times\id)_*\circ\Theta_g}\\
&\hskip-50pt\pi_{M\times(\A^1\setminus\{0\})!}\circ\Sigma^{p_1^*T_{M}}(1_{M\times(\A^1\setminus\{0\})})
\ar[dd]\\
\pi_{Def(i')!}\circ \Sigma^{p_{M'}^*p_1^*T_{M'}}(1_{Def(i')})\ar[dd]\ar[rd]^{\hskip30pt
Def(g)_*\circ\Theta_g}\\&
\pi_{Def(i)!}\circ \Sigma^{p_M^*p_1^*T_{M}}(1_{Def(i)})\ar[dd]\\
\mfC_{i'}(\sigma_{i'}^*T_{M'})_\BM \ar[dr]^{\hskip20pt\psi(g)=\mfC(g)_*\circ\Theta_g}\\
&\mfC_{i}(\sigma_i^* T_{M})_\BM
}
\]
which in turn gives the commutative diagram
\[
\xymatrixcolsep{50pt}
\xymatrix{
\Cst_Z\ar[d]_{\alpha_{i'}}\ar[dr]^{\alpha_i}\\
\mfC_{i'}(\sigma_{i'}^*T_{M'})_\BM\ar[r]^{\psi_g}\ar[d]_{\del}&
\mfC_{i}(\sigma_i^* T_{M})_\BM\ar[d]^\del\\
\Sigma_{\G_m}^{-1} M'\times(\A^1\setminus\{0\})_+\ar[r]^{\Sigma_{\G_m}^{-1}g\times\id}\ar[dr]_{\bar{p}_2^{M'}}&
\Sigma_{\G_m}^{-1} M\times(\A^1\setminus\{0\})_+\ar[d]^{\bar{p}_2^M}\\
&\mS_B,
}
\]
completing the proof. 
\end{proof}

This lemma allows us to make the following definition.

\begin{definition}\label{def:FundClass} Let $Z$ be in $\Sch^G/B$. The {\em fundamental class} $[\Cst_Z]\in \mS_B^{0,0}(\Cst_Z)$ is  the element represented by the map
$[\mfC_i]\circ\alpha_i:\Cst_Z\to   \mS_B$
for any choice of closed immersion $i:Z\to M$ in $\Sch^G/B$ with $M$ smooth over $B$.

If $\sE$ is a commutative monoid in $\SH^G(B)$ with unit $\epsilon_\sE:\mS_B\to \sE$, we define the fundamental class $[\Cst_Z]_\sE\in \sE^{0,0}(\Cst_Z)$ by composing $[\Cst_Z]$ with $\epsilon_\sE$.
\end{definition}

\section{The virtual fundamental class for a reduced normalized representative}\label{sec:RedNormVirClass}

Let $\L_{Z/B}$ be the relative cotangent complex on $Z\in\Sch^G/B$ and let $[\phi]:E_\bullet\to \L_{Z/B}$ be a perfect obstruction theory on $Z$. Recall that we use homological notation for complexes.

If we choose  a closed immersion $i:Z\to M$ in $\Sch^G/B$ with $M\in \Sm^G/B$, we  have the explicit model $(\sI_Z/\sI_Z^2\xrightarrow{d}i^*\Omega_{M/B})$ for $\tau_{\le 1}\L_{Z/B}$.  Since $E_\bullet$ is by definition supported in $[0,1]$ and $Z$ satisfies the $G$-resolution property,  we have a global resolution, that is,  we have a two-term complex of locally free sheaves in $\Coh_Z^G$, $F_\bullet:=(F_1\to F_0)$, and a map of complexes
\[
\phi:(F_1\xrightarrow{d_F} F_0)\to (\sI_Z/\sI_Z^2\xrightarrow{d}i^*\Omega_{M/B})
\]
which induces an isomorphism on $h_0$ and a surjection on $h_1$,   
representing $\tau_{\le1}$ applied to $[\phi]:E_\bullet\to  \L_{Z/B}$.  

\begin{definition} We call  a representative $(F_\bullet, \phi)$ of $[\phi]$ as above a {\em normalized} representative if the maps $\phi_0, \phi_1$ are surjective.   If in addition $F_0=i^*\Omega_{M/B}$ and $\phi_0$ is the identity, we call $\phi$ {\em reduced}. 
\end{definition}

Let $i:Z\to M$ be a closed immersion in $\Sch^G/B$ with $M$ smooth over $B$ and let
\[
\phi:(F_1\to F_0=i^*\Omega_{M/B})\to (\sI_Z/\sI_Z^2\to i^*\Omega_{M/B}).
\]
be a reduced normalized representative of a perfect obstruction theory $[\phi]$ on $Z$.  The surjection $\phi_1:F_1\to \sI_Z/\sI_Z^2$ induces a surjection of graded algebras $\Sym^*F_1\to \oplus_n  \sI_Z^n/\sI_Z^{n+1}$, and thereby a closed immersion of cones over $Z$
\[
i_\phi:\mfC_{i}\to F^1:=\V(F_1).
\]
Let $p_{F^1}:F^1\to Z$ be the projection and let $0_{F^1}:Z\to F^1$ be the 0-section.

The map $i_\phi$ induces the proper push-forward
\[
i_{\phi*}:\mS_B^\BM(\mfC_{i}, \sigma_i^*T_M)\to \mS_B^\BM(F^1,p_{F^1}^*i_Z^*T_M)
\]
Noting that  $T_{F^1/Z}=p_{F^1}^*(F^1)$,  the zero section $0_{F^1}$ gives us the Gysin map
\[
0_{F^1}^!:\mS_B^\BM(F^1, p_{F^1}^*i^*T_M)\to \mS_B^\BM(Z, i^*T_M-F^1)
\]

We have the fundamental class $[\mfC_i]\in \mS_B^\BM(\mfC_{i}, \sigma_i^*T_M)$ (Definition~\ref{def:FundClass0}).
\begin{definition} \label{defVFCRN}The {\em virtual fundamental class} for a reduced normalized representative $\phi$ of a perfect obstruction theory $[\phi]:E_\bullet\to \L_{Z/B}$ with respect to a closed immersion $i$ is defined as
\[
[Z,\phi, i]^\vir:=0_{F^1}^!(i_{\phi*}([\mfC_i]))\in \mS^\BM_B(Z, i^*T_M-F^1)=
\mS^\BM_B(Z,\V(E_\bullet))
\]
\end{definition}
The identity $ \mS^\BM_B(Z, i^*T_M-F^1)=
\mS^\BM_B(Z,\V(E_\bullet))$ is really the isomorphism induced by the given isomorphism $(F_1\to F_0)\cong E_\bullet$  in $D^\perf_G(Z)$.

\section{Reduced normalized representatives of a perfect obstruction theory}\label{sec:ObstThy}
We show how a given perfect obstruction theory $[\phi]:E_\bullet\to \L_{Z/B}$
 on some $Z\in \Sch^G/B$ admits a reduced normalized representative ``up to $\A^1$-homotopy equivalence'', in a way which will be clarified in this section. If $Z$ is affine, $[\phi]$ already admits 
a reduced normalized representative. In general we replace $Z$ with a Jouanolou cover $p_Z:\tilde{Z}\to Z$ and we construct an ``induced perfect obstruction theory'' $p_Z^![\phi]:p_Z^!E_\bullet\to \L_{\tilde{Z}/B}$ on $\tilde{Z}$.

 As in the previous section, let $\L_{Z/B}$ be the relative cotangent complex on $Z$ and let $[\phi]:E_\bullet\to \L_{Z/B}$ be a perfect obstruction theory on $Z$. Choose  a closed immersion $i:Z\to M$ in $\Sch^G/B$ with $M\in \Sm^G/B$, a two-term complex of locally free sheaves in $\Coh_Z^G$, $F_\bullet:=(F_1\to F_0)$, and a map of complexes
\[
\phi:(F_1\xrightarrow{d_F} F_0)\to (\sI_Z/\sI_Z^2\xrightarrow{d}i^*\Omega_{M/B})
\]
which induces an isomorphism on $h_0$ and a surjection on $h_1$,   
representing $\tau_{\le1}$ applied to $[\phi]:E_\bullet\to  \L_{Z/B}$.  

By the $G$-resolution property for $Z$, there is a locally free sheaf $\sF$ on $Z$ and a surjection $p:\sF\to \sI_Z/\sI_Z^2$. We may then replace $(F_\bullet,\phi)$ by 
\[
F_1\oplus \sF\xrightarrow{d_F\oplus \id_\sF}F_0\oplus \sF
\]
and map the copy of $\sF$ in degree 0 to $i^*\Omega_{M/B}$ by $d\circ p$, giving the map
\[
\phi':(F_1\oplus \sF\xrightarrow{d_F\oplus \id_\sF}F_0\oplus \sF)\to 
(\sI_Z/\sI_Z^2\xrightarrow{d}i^*\Omega_{M/B}).
\]
The map  $\phi'_1:=\phi+p$ is surjective by construction and the assumption that $h_0(\phi)$ is an isomorphism implies that $\phi'_0$ is surjective as well. Thus, each $[\phi]$ admits a normalized representative.

For a normalized representative $\phi:F_\bullet\to (\sI_Z/\sI_Z^2\xrightarrow{d}i^*\Omega_{M/B})$, we let $K_i\subset F_i$ be the kernel of $\phi_i$. We let $F^i\to Z$ be the   vector bundle 
$\V(F_i):=\Spec_{\sO_Z}\Sym^* F_i$ and similarly define $K^i:=\V(K_i)$. 

\begin{lemma} \label{lem:ExactSeq} Let $\phi:F_\bullet\to (\sI_Z/\sI_Z^2\xrightarrow{d}i^*\Omega_{M/B})$ be a normalized representative of a perfect obstruction theory $[\phi]$. Let $K(h_1(F_\bullet))$ be the kernel of the surjection $h_1(\phi):h_1(F_\bullet)\to h_1(\L_{Z/B})$. Then $K_0$ is locally free and  in the commutative diagram
\[
\xymatrix{
&0\ar[d]&0\ar[d]&0\ar[d]\\
0\ar[r]&K(h_1(F_\bullet))\ar[r]\ar[d]&K_1\ar[r]\ar[d]&K_0\ar[r]\ar[d]&0\ar[d]\\
0\ar[r]&h_1(F_\bullet)\ar[r]\ar[d]_{h_1(\phi)}&F_1\ar[r]^{d_F}\ar[d]_{\phi_1}&F_0\ar[d]_{\phi_0}\ar[r]&h_0(F_\bullet)\ar[d]_{h_0(\phi)}\ar[r]&0\\
0\ar[r]&h_1(\L_{Z/B})\ar[r]\ar[d]&\sI_Z/\sI_Z^2\ar[r]_d\ar[d]&i^*\Omega_{M/B}\ar[r]\ar[d]&h_0(\L_{Z/B})\ar[r]\ar[d]&0\\
&0&0&0&0
}
\]
all the rows and columns are exact.
\end{lemma}

\begin{proof} Our assumption that $\phi$ is a normalized representative is just that $\phi_0, \phi_1$ and $h_1(\phi)$ are surjective, and $h_0(\phi)$ is an isomorphism. The rest follows by the snake lemma.
\end{proof}

\begin{lemma}\label{lem:ReducedNorm}
Suppose $Z$ is affine, choose a closed immersion $i:Z\to M$ in $\Sch^G/B$ with $M\in \Sm^G/B$ and let $[\phi]$ be a perfect obstruction theory on $Z$. Then  $[\phi]$ admits a reduced normalized representative. 
\end{lemma}

\begin{proof} We have already seen that  $[\phi]$ admits a  normalized representative 
\[
\phi:(F_1\to F_0)\to (\sI_Z/\sI_Z^2\to i^*\Omega_{M/B}).
\]

We use the notation of Lemma~\ref{lem:ExactSeq}. Since $Z$ is affine, each locally free  $\sF$ in $\QCoh_Z^G$ is a projective object in $\QCoh_Z^G$. Thus, we may choose a splitting $s_K:K_0\to K_1$ to the surjection $K_1\to K_0$. This gives us the commutative diagram
\[
\xymatrix{
K_0\ar@{=}[d]\ar@/^20pt/[rr]^{i_1}\ar[r]^{s_K}&K_1\ar[r]\ar[d]&F_1\ar[d]^{d_F}\ar[r]^{\phi_1}&\sI_Z/\sI^2_Z\ar[d]^d\\
K_0\ar@{=}[r]&K_0\ar[r]&F_0\ar[r]_{\phi_0}&i^*\Omega_{M/B}
}
\]
Replacing $F_1$ with $F_1':=F_1/i_1(K_0)$ and $F_0$ with $i^*\Omega_{M/B}\cong F_0/K_0$,  we have the reduced normalized representative 
\[
(F'_1\xrightarrow{d_{F'}}i^*\Omega_{M/B})\xrightarrow{(\phi'_1,\id)} (\sI_Z/\sI_Z^2\to i^*\Omega_{M/B})
\]
for $[\phi]$. 
\end{proof}

\begin{lemma}\label{lem:RedNormIso} Suppose $Z$ is affine and choose a closed immersion $i:Z\to M$ in $\Sch^G/B$ with $M\in \Sm^G/B$. If $\phi:F_\bullet\to (\sI_Z/\sI_Z^2\to i^*\Omega_{M/B})$ and 
$\phi':F'_\bullet\to (\sI_Z/\sI_Z^2\to i^*\Omega_{M/B})$ are two reduced normalized  representatives of a given perfect obstruction theory $[\phi]:E_\bullet\to \L_{Z/B}$ on $Z$, then the induced isomorphism $F_\bullet\cong E_\bullet\cong F'_\bullet$ in $D^\perf_G(Z)$ arises from an isomorphism of complexes $\rho_\bullet:F_\bullet\to F'_\bullet$ making the diagram
\[
\xymatrix{
F_\bullet\ar[r]^{\rho_\bullet}\ar[dr]&F'_\bullet\ar[d]\\
&(\sI_Z/\sI_Z^2\to i^*\Omega_{M/B})
}
\]
commute. Moreover $\rho_\bullet$ is unique up to chain homotopy of the form $h:i^*\Omega_{M/B}\to h_1(F'_\bullet)\subset F_1'$ satisfying $d_{F'_\bullet}h=0=\phi_1'hd_{F_\bullet}$.
\end{lemma} 

\begin{proof} We denote the homology of the complex $(\sI_Z/\sI_Z^2\to i^*\Omega_{M/B})$ by $h_1, h_0$.  As we are given an isomorphism  $F_\bullet\cong  F'_\bullet$ in $D^\perf_G(Z)$, we have isomorphisms $\xi_i:h_i(F_\bullet)\cong h_i(F_\bullet')$ in $\Coh_Z^G$  such that the induced map on the Ext groups
\[
\Ext_{\Coh_Z^G}^2(h_0(F_\bullet), h_1(F_\bullet))\to
\Ext_{\Coh_Z^G}^2(h_0(F'_\bullet), h_1(F'_\bullet))
\]
sends the class of $0\to h_1(F_\bullet)\to F_1\to F_0\to h_0(F_\bullet)\to 0$ to the class of 
$0\to h_1(F'_\bullet)\to F'_1\to F'_0\to h_0(F'_\bullet)\to 0$. 

Let $M_0\subset i^*\Omega_{M/B}$ the image subsheaf $d(\sI_Z/\sI_Z^2)$. The maps
$\phi$, $\phi'$ give rise to exact sequences
\begin{equation}\label{eqn:Ext1}
0\to h_1(F_\bullet)\to F_1\xrightarrow{d_F}M_0\to 0,
\end{equation}
\begin{equation}\label{eqn:Ext1'}
0\to h_1(F'_\bullet)\to F_1'\xrightarrow{d_{F'}}M_0\to 0
\end{equation}
and
\begin{equation}\label{eqn:Ext1''}
0\to M_0\to i^*\Omega_{M/B}\to h_0\to 0
\end{equation}
Thus the $\Ext^2$ classes considered above arise from  the $\Ext^1$-classes of \eqref{eqn:Ext1} and \eqref{eqn:Ext1'} respectively by taking the boundary in the long exact $\Ext$ sequence associated to \eqref{eqn:Ext1''}. Since $Z$ is affine, $\Ext^1_{\Coh_Z^G}(i^*\Omega_{M/B}, -)=0$, so the isomorphism $h_1(F_\bullet)\cong h_1(F'_\bullet)$ sends the $\Ext^1$-class of \eqref{eqn:Ext1} to that of \eqref{eqn:Ext1'} giving us the isomorphism $\rho_1:F_1\to F_1'$ making the diagram 
\[
\xymatrix{
h_1(F_\bullet)\ar[r]\ar[d]_{\xi_1}&F_1\ar[d]_{\rho_1}\ar[r]&M_0\ar@{=}[d]\\
h_1(F'_\bullet)\ar[r]&F_1'\ar[r]&M_0
}
\]
commute. This gives us the isomorphism $\rho:=(\rho_1,\id):F_\bullet\to F_\bullet'$ of perfect complexes. 

Since $Z$ is affine and $G$ is tame, $\Coh_Z^G$ has enough projectives, so   $D^-(\Coh_Z^G)$ is equivalent to the bounded-above homotopy category of the full subcategory of projective objects in $\Coh_Z^G$; the projective objects are the $\sF\in  \Coh_Z^G$ which are locally free as $\sO_Z$-modules. 
 
Consider the two morphisms $\phi, \phi'\circ\rho:F_\bullet\to (\sI_Z/\sI_Z^2\to i^*\Omega_{M/B})$. We have $[\phi]=[\phi'\circ\rho]$ as maps in $D^b(\Coh_Z^G)$, and as $F_\bullet$ is a perfect complex, the maps $\phi, \phi'\circ\rho$ are chain homotopic. Since $\phi'_1$ is surjective, for $h:F_0\to \sI_Z/\sI_Z^2$ a chain homotopy, there is a $\tilde{h}:F_0\to F'_1$ with $
\phi'_1\circ\tilde{h}=h$. Since $\phi_0=\id_{i^*\Omega_{M/B}}$, we have 
\[
0=d\circ h=d\circ\phi'_1\circ\tilde{h}=\id_{i^*\Omega_{M/B}}\circ d_{F'_\bullet}\circ\tilde{h}
\]
so we may consider $\tilde{h}$ as a map $\tilde{h}:i^*\Omega_{M/B}\to h_1(F'_\bullet)$. Replacing $\rho$ with $\rho':=\rho-d\tilde{h}$, we see that $\rho':F_\bullet\to F'_\bullet$ satisfies $\phi'\circ\rho'=\phi$ as maps of complexes. Changing notation, we may assume that $\rho$ is an isomorphism satisfying $\phi'\circ\rho=\phi$.

Given two isomorphisms $\rho_1, \rho_2:F_\bullet\to F'_\bullet$  over $\L_{Z/B}$, both representing the same morphism in  $D^\perf_G(Z)$, then $\tau_{\le1}\rho_1$ and $\tau_{\le1}\rho_2$ are chain homotopic by a map $h:F_0=i^*\Omega_{M/B}\to F_1'$. The fact that both $\tau_{\le1}\rho_i$ are isomorphisms over $(\sI_Z/\sI_Z^2\to i^*\Omega_{M/B})$ implies the identities $d_{F'_\bullet}h=0=\phi_1'hd_{F_\bullet}$, the first of which shows that $h(F_0)\subset h_1(F'_\bullet)$.  
\end{proof}

We return to case of  arbitrary $Z\in \Sch^G/B$ with a closed immersion $i_Z:Z\to M$ in $\Sch^G/B$,  $M\in \Sm^G/B$. Suppose we have a smooth morphism $p_M:\tilde{M}\to M$. Form the cartesian square
\[
\xymatrix{
\tilde{Z}\ar[r]^{i_{\tilde{Z}}}\ar[d]_{p_Z}&\tilde{M}\ar[d]^{p_M}\\
Z\ar[r]_{i_Z}&M
}
\]
Since $p_Z$ is smooth, we have the distinguished triangle in $D^\perf_G(\tilde{Z})$
\begin{equation}\label{eqn:InducedDistTriangle1}
p_Z^*\L_{Z/B}\xrightarrow{\L(p_Z)} \L_{\tilde{Z}/B}\to \Omega_{\tilde{Z}/Z}\to 
p_Z^*\L_{Z/B}[1]
\end{equation}
and the isomorphism of locally free sheaves on $\tilde{Z}$, $\Omega_{\tilde{Z}/Z}\cong i_{\tilde{Z}}^*\Omega_{\tilde{M}/M}$.  Applying the truncation functor $h_0$ gives us the exact sequence of sheaves
\[
0\to p_Z^*\Omega_{Z/B}\to \Omega_{\tilde{Z}/B}\xrightarrow{\pi} \Omega_{\tilde{Z}/Z}\to0
\]

Given a perfect obstruction theory $[\phi]:E_\bullet\to  \L_{Z/B}$, we say that a perfect obstruction theory $[\tilde{\phi}]:\tilde{E}_\bullet\to \L_{\tilde{Z}/B}$ is {\em induced from } $[\phi]$ if \\[5pt]
i.   $\tilde{E}_\bullet=p_Z^*E_\bullet\oplus \Omega_{\tilde{Z}/Z}$\\[2pt]
ii. $[\tilde{\phi}]$ fits into a commutative diagram
\[
\xymatrix{
p_Z^*E_\bullet \ar[r]^i\ar[d]_{p_Z^*[\phi]}&\tilde{E}_\bullet\ar[d]^{[\tilde{\phi}]}\ar[r]^p&\Omega_{\tilde{Z}/Z}\ar@{=}[d]\\
p_Z^*\L_{Z/B}\ar[r]^{\L(p_Z)}& \L_{\tilde{Z}/B}\ar[r]& \Omega_{\tilde{Z}/Z}
}
\]
with  $i$ and $p$ the canonical inclusion and projection. 

Since $h_0([\phi]):h_0(E_\bullet)\to \Omega_{Z/B}$ and $h_0([\tilde{\phi}]):h_0(\tilde{E}_\bullet)\to \Omega_{\tilde{Z}/B}$ are both isomorphisms, a necessary condition for an induced obstruction theory to exist is that the surjection $\Omega_{\tilde{Z}/B}\to \Omega_{\tilde{Z}/Z}$ should split.

\begin{lemma}\label{lem:InducedObstThy} Suppose $\tilde{Z}$ is affine. Then for each  perfect obstruction theory $[\phi]$ on $Z$, there exists an induced obstruction theory $[\tilde{\phi}]:\tilde{E}_\bullet\to \L_{\tilde{Z}/B}$ on $\tilde{Z}$. 

Moreover,  $[\tilde{\phi}]$ is unique up to canonical isomorphism $\alpha:(\tilde{E},[\tilde{\phi}])\xrightarrow{\sim}(\tilde{E},[\tilde{\phi}'])$ of obstruction theories.
\end{lemma}

\begin{proof}  
Since $\tilde{Z}=\Spec A$ is affine and $G$ is tame, $\Hom_{D^\perf_G(\tilde{Z})}(\Omega_{\tilde{Z}/Z}, p_Z^*\L_{Z/B}[1])=0$. This gives us a splitting $\tilde{s}$ of the distinguished triangle
\[
p_Z^*\L_{Z/B}\to \L_{\tilde{Z}/B}\xrightarrow{\pi} \Omega_{\tilde{Z}/Z}\to p_Z^*\L_{Z/B}[1]
\]
and thereby an isomorphism $\L_{\tilde{Z}/B}\cong p_Z^*\L_{Z/B} \oplus \Omega_{\tilde{Z}/Z}$ in $D^\perf_G(\tilde{Z})$. From this we see that each perfect obstruction theory $[\phi]:E_\bullet\to 
\L_{Z/B}$ admits an induced perfect obstruction theory of the form
\[
[\tilde\phi_{\tilde{s}}]:=[\phi]\oplus\id_{\Omega_{\tilde{Z}/Z}}:p_Z^*E_\bullet \oplus \Omega_{\tilde{Z}/Z}\to p_Z^*\L_{Z/B} \oplus \Omega_{\tilde{Z}/Z}\cong \L_{\tilde{Z}/B} 
\]
with the isomorphism depending on the choice of splitting $\tilde{s}$ of the distinguished triangle

We have the distinguished triangle
\[
\tau_{\ge1}\L_{\tilde{Z}/B}\to \L_{\tilde{Z}/B}\to h_0\L_{\tilde{Z}/B}\to (\tau_{\ge1}\L_{\tilde{Z}/B})[1]
\]
Since $\tilde{Z}$ is affine,  $\Omega_{\tilde{Z}/Z}$ is locally free and $h_n\tau_{\ge1}\L_{\tilde{Z}/B}=0$ for $n\le0$ and for $n>>0$, we have $\Hom_{D^\perf_G(\tilde{Z})}(\Omega_{\tilde{Z}/Z}, \tau_{\ge1}\L_{\tilde{Z}/B}[i])=0$ for all $i\ge0$, so each map of sheaves
\[
s:\Omega_{\tilde{Z}/Z}\to h_0\L_{\tilde{Z}/B}=\Omega_{\tilde{Z}/B}
\]
lifts uniquely to a map $\tilde{s}:\Omega_{\tilde{Z}/Z}\to \L_{\tilde{Z}/B}$ in $D^\perf_G(\tilde{Z})$. This shows that  splittings of the map  $\pi$ are in bijection with splittings of the surjection $\Omega_{\tilde{Z}/B}\to \Omega_{\tilde{Z}/Z}$.

Since $[\phi]$ is a perfect obstruction theory, we also have $h_0(p_Z^*E_\bullet)=p_Z^*\Omega_{Z/B}$, and the same argument as above show that a second splitting $s'$ of $\Omega_{\tilde{Z}/B}\to \Omega_{\tilde{Z}/Z}$ determines a map $\lambda:\Omega_{\tilde{Z}/Z}\to p_Z^*E_\bullet$.   The map $\alpha:p_Z^*E_\bullet \oplus \Omega_{\tilde{Z}/Z}\to
p_Z^*E_\bullet \oplus \Omega_{\tilde{Z}/Z}$ with matrix
\[
\alpha=\begin{pmatrix}\id&\lambda\\0&\id\end{pmatrix}
\]
thus defines an automorphism of $p_Z^*E_\bullet \oplus \Omega_{\tilde{Z}/Z}$ intertwining $[\tilde{\phi}_{\tilde{s}}]$ and $[\tilde{\phi}_{\tilde{s}'}]$.
\end{proof}

Assuming that $\tilde{Z}$ is affine as above, we write $p_Z^![\phi]$ for the (unique up to canonical isomorphism) perfect obstruction theory induced by a given perfect obstruction theory  $[\phi]$.

 \begin{remark}[Jouanolou covers]\label{rem:Jouanolou} Hoyois \cite[Proposition 2.20]{Hoyois6} has shown that the Jouanolou trick extends to the equivariant case: for each $M\in \Sch^G/B$, there is an affine space bundle $\tilde{M}\to M$ in $\Sch^G/B$ such that $\pi_{\tilde{M}}:\tilde{M}\to B$ is an affine morphism. We call such a map $\tilde{M}\to M$ a {\em Jouanolou cover} of $M$. For affine $B$,  Lemma~\ref{lem:HomotopyInv} will thus enable us to reduce various constructions to the case of affine $Z\in\Sch^G/B$.
\end{remark}

We now assume that $B$ is affine. Using Remark~\ref{rem:Jouanolou},  for each $Z\in \Sch^G/B$ and each closed immersion $Z\to M$ with $M\in \Sm^G/B$, there is an affine space bundle $\tilde{M}\to M$ with $\tilde{M}$ affine, and thus $\tilde{Z}:=Z\times_M\tilde{M}$ is affine as well.

 \begin{lemma}\label{lem:Discussion} Suppose $B$ is affine and  we have a closed immersion $i_Z:Z\to M$ with $M$ in $\Sm^G/B$. Let $\tilde{M}\to M$ be a Jouanolou cover of $M$ and form the cartesian square
\[
\xymatrix{
\tilde{Z}\ar[r]^{i_{\tilde{Z}}}\ar[d]_{p_Z}&\tilde{M}\ar[d]^{p_M}\\
Z\ar[r]_{i_Z}&M
}
\]
Let $(E_\bullet, [\phi])$ be a perfect obstruction theory on $Z$. Then an induced obstruction theory $p_Z^![\phi]$ on $\tilde{Z}$ exists, is unique up to canonical isomorphism,  and  admits a reduced normalized representative $\tilde\phi$ for $p_Z^![\phi]$.
\end{lemma}

\begin{proof} Since $\tilde{Z}$ is affine, this follows from  Lemma~\ref{lem:ReducedNorm} and Lemma~\ref{lem:InducedObstThy}.
\end{proof}

\section{The virtual fundamental class-the general case}\label{sec:VFCGeneral}
In this section, we assume that $B$ is affine.

We use our construction in \S\ref{sec:RedNormVirClass} of the virtual fundamental class  for a reduced normalized representative together with the results of \S\ref{sec:ObstThy} to give a well-defined virtual fundamental class for  a perfect obstruction theory on a general $Z\in  \Sch^G/B$.
 
Let  $[\phi]:E\to \L_{Z/B}$ be a perfect obstruction theory on some   $Z\in \Sch^G/B$. Choose a closed immersion $i:Z\to M$ with $M\in\Sm^G/B$ and take a Jouanolou cover 
\[
\xymatrix{
\tilde{Z}\ar[r]^{i_{\tilde{Z}}}\ar[d]_{p_Z}&\tilde{M}\ar[d]^{p_M}\\
Z\ar[r]_{i_Z}&M
}
\]
By Lemma~\ref{lem:Discussion}, we have an induced perfect obstruction theory $(\tilde{E}_\bullet, [\tilde\phi]):=(p_Z^!E_\bullet, p_Z^![\phi])$ on $\tilde{Z}$ and  a reduced normalized representative $\tilde\phi:\tilde{F}_\bullet\to \tau_{\le 1}\L_{\tilde{Z}/B}$ for $[\tilde\phi]$. In particular, we have the identity
\[
 \tilde{E}_\bullet= p_Z^*E_\bullet\oplus \Omega_{\tilde{Z}/Z}
\]
inducing an isomorphism of suspension operators 
\[
\Sigma^{p_Z^*\V(E_\bullet)+T_{\tilde{Z}/Z}}\cong \Sigma^{\V(\tilde{E}_\bullet)}. 
\]
 The quasi-isomorphism $\tilde{F}_\bullet\to\tilde{E}_\bullet$, and the fact that 
$\tilde{F}_\bullet$ is reduced and normalized  gives us a canonical isomorphism
\[
\Sigma^{i_{\tilde{Z}}^*T_{\tilde{M}}-\tilde{F}^1}\cong 
\Sigma^{\V(\tilde{E}_\bullet)}.
\]

Since $p_Z:\tilde{Z}\to Z$ is an affine space bundle, homotopy invariance implies that the smooth push-forward map
\[
p_{Z*}:\pi_{\tilde{Z}!}\circ \Sigma^{p_Z^*\V(E_\bullet)+T_{\tilde{Z}/Z}}\circ    p_Z^*\to 
\pi_{Z!}\circ  \Sigma^{\V(E_\bullet)}
\]
is an isomorphism. Applying this to $1_Z$, using the isomorphisms 
\[
 \Sigma^{p_Z^*\V(E_\bullet)+T_{\tilde{Z}/Z}}\cong 
 \Sigma^{\V(\tilde{E}_\bullet)}\cong \Sigma^{i_{\tilde{Z}}^*T_{\tilde{M}}-\tilde{F}^1}
 \]
 described above and taking the inverse gives us the isomorphism
\begin{equation}\label{eqn:JouanolouComp}
\vartheta_{i_Z, p_Z, \tilde{F}_\bullet}: Z(\V(E_\bullet))_\BM\to 
\tilde{Z}(i_{\tilde{Z}}^*T_{\tilde{M}}-\tilde{F}^1)_\BM.
\end{equation}

We have the virtual fundamental class $[\tilde{Z}, \tilde{\phi}, i_{\tilde{Z}}]^\vir\in \mS^\BM_B(\tilde{F}, i_{\tilde{Z}}^*T_{\tilde{M}}-\tilde{F}^1)$ (Definition~\ref{defVFCRN}), giving us the class
\[
\vartheta_{i_Z, p_Z, \tilde{F}_\bullet}^*([\tilde{Z}, \tilde{\phi}, i_{\tilde{Z}}]^\vir)\in \mS^\BM_B(F, \V(E_\bullet)).
\]
The main point is that this class is independent of the various choices made. This is proven with the help of the following two results.

\begin{lemma}\label{lem:BundleInv} Let $p_M:\tilde{M}\to M$ be a Jouanolou cover and let $q_M:\hat{M}\to \tilde{M}$ be a vector bundle. Form the cartesian diagram
\begin{equation}\label{eqn:DiagBundleInv}
\xymatrix{
\hat{Z}\ar@/_1cm/[dd]_{\hat{p}_Z}\ar[r]^{i_{\hat{Z}}}\ar[d]_{q_Z}&\hat{M}\ar[d]^{q_M}\ar@/^1cm/[dd]^{\hat{p}_M}\\
\tilde{Z}\ar[r]^{i_{\tilde{Z}}}\ar[d]_{p_Z}&\tilde{M}\ar[d]^{p_M}\\
Z\ar[r]_{i_Z}&M
}
\end{equation}
Let $[\phi]:E_\bullet\to \L_{Z/B}$ be a perfect obstruction theory, form the induced obstruction theories $(p_Z^!E_\bullet, p_Z^![\phi])$ and $(\hat{p}^!E_\bullet, \hat{p}_Z^![\phi])$, and  let  
\[
\tilde{\phi}:\tilde{F}_\bullet\to (\sI_{\tilde{Z}}/\sI_{\tilde{Z}}^2\xrightarrow{d}i_{\tilde{Z}}^*\Omega_{\tilde{M}/B}),\ \hat{\phi}:\hat{F}_\bullet\to  (\sI_{\hat{Z}}/\sI_{\hat{Z}}^2\xrightarrow{d}i_{\hat{Z}}^*\Omega_{\hat{M}/B})
\]
be reduced normalized representatives of $p_Z^![\phi]$ and $\hat{p}_Z^![\phi]$, respectively. We have the isomorphisms of Lemma~\ref{lem:HomotopyInv},
\[
\Cst(p_M):\Cst_{\tilde{Z}}\to \Cst_Z,\ \Cst(\hat{p}_M):\Cst_{\hat{Z}}\to \Cst_Z. 
\]

The respective 0-sections $0_{\tilde{F}^1}:\tilde{Z}\to \tilde{F}^1$, $0_{\hat{F}^1}:\hat{Z}\to \hat{F}^1$ induce the Gysin push-forward maps  $0_{\tilde{F}^1!}$ and $0_{\hat{F}^1!}$, and the respective closed immersions $i_{\tilde{\phi}}:\mfC_{i_{\tilde{Z}}}\to \tilde{F}^1$ and $i_{\hat{\phi}}:\mfC_{i_{\hat{Z}}}\to \hat{F}^1$ induce the proper pull-back maps
 $i^*_{\tilde{\phi}}$ and $i^*_{\hat{\phi}}$. We have as well the isomorphisms
  $\alpha_{i_{\tilde{Z}}}:\Cst_{\tilde{Z}}\to \mfC_{i_{\tilde{Z}}}(\sigma_{\tilde{i}}^*T_{\tilde{M}})_\BM$ and  $\alpha_{i_{\hat{Z}}}:\Cst_{\hat{Z}}\to\mfC_{i_{\tilde{Z}}}(\sigma_{\hat{i}}^*T_{\hat{M}})_\BM$

Define maps $s_{\tilde\phi}:\tilde{Z}(i_{\tilde{Z}}^*T_{\tilde{M}}-\tilde{F}^1)_\BM\to \Cst_{\tilde{Z}}$, $s_{\hat\phi}:\hat{Z}(i_{\hat{Z}}^*T_{\hat{M}}-\hat{F}^1)_\BM\to \Cst_{\hat{Z}}$ by
\begin{equation}\label{eqn:CompMapDef}
s_{\tilde\phi}:=\alpha_{i_{\tilde{Z}}}^{-1}\circ i^*_{\tilde{\phi}}\circ 0_{\tilde{F}^1!}\text{ and }
s_{\hat\phi}:=\alpha_{i_{\hat{Z}}}^{-1}\circ i^*_{\hat{\phi}}\circ 0_{\hat{F}^1!}. 
\end{equation}
 Suppose that $\hat{F}_1=q_Z^*\tilde{F}_1$. Then 
\[
\Cst(p_M)\circ s_{i_{\tilde{\phi}}}\circ\vartheta_{i_Z,p_Z, \tilde{F}_\bullet}=\Cst(\hat{p}_M)\circ s_{i_{\hat{\phi}}}\circ \vartheta_{i_Z, \hat{p}_Z, \hat{F}_\bullet}
\]
\end{lemma}

\begin{proof} The identity $\hat{F}_1=q_Z^*\tilde{F}_1$ gives us the map of vector bundles over $q_Z$, $q_F:\hat{F}^1\to \tilde{F}^1$. This identifies $\hat{F}^1$ with the vector bundle $\hat{Z}\times_{\tilde{Z}}\tilde{F}^1\to \tilde{F}^1$, which we denote by $q_F:V\to \tilde{F}^1$.

Saying that $(\tilde{F}_\bullet, \tilde\phi)$ represents $p_Z^![\phi]$ gives us an isomorphism $\tilde{F}_\bullet\cong p_Z^*E_\bullet\oplus \Omega_{\tilde{Z}/Z}$; similar we are given an isomorphism $\hat{F}_\bullet\cong \hat{p}_Z^*E_\bullet \oplus \Omega_{\hat{Z}/Z}$. These give us isomorphisms
\begin{gather*}
\tilde\alpha:\tilde{Z}(i^*_{\tilde{Z}}T_{\tilde{M}}-\tilde{F}^1)_\BM\xrightarrow{\sim}
\tilde{Z}(p_Z^*\V(E_\bullet)+T_{\tilde{Z}/Z})_\BM,\\
\hat\alpha:\hat{Z}(i^*_{\hat{Z}}T_{\hat{M}}-\hat{F}^1)_\BM\xrightarrow{\sim}
\hat{Z}(\hat{p}_Z^*\V(E_\bullet)+T_{\hat{Z}/Z})_\BM
\end{gather*}
The exact sequence 
\[
0\to T_{\hat{Z}/\tilde{Z}}\to T_{\hat{Z}/Z}\to q_Z^*T_{\tilde{Z}/Z}\to0
\]
gives the isomorphism
\[
\alpha:\hat{Z}(\hat{p}_Z^*\V(E_\bullet)+T_{\hat{Z}/Z})_\BM\to
\hat{Z}(q_Z^*(p_Z^*\V(E_\bullet)+ T_{\tilde{Z}/Z}) +T_{\hat{Z}/\tilde{Z}})_\BM
\]
Since $q_Z:\hat{Z}\to \tilde{Z}$ is a vector bundle over $\tilde{Z}$, the smooth push-forward map
\[
q_{Z*}:\hat{Z}(q_Z^*(p_Z^*\V(E_\bullet)+T_{\tilde{Z}/Z})+T_{\hat{Z}/\tilde{Z}})_\BM\to
\tilde{Z}(p_Z^*\V(E_\bullet)+T_{\tilde{Z}/Z})_\BM
\]
is an isomorphism.  Putting these together gives us the isomorphism
\[
\beta:\hat{Z}(i^*_{\hat{Z}}T_{\hat{M}}-\hat{F}^1)_\BM\xrightarrow{\sim}
\tilde{Z}(i^*_{\tilde{Z}}T_{\tilde{M}}-\tilde{F}^1)_\BM
\]
defined as $\beta:=(\tilde\alpha)^{-1}\circ q_{Z*}\circ \alpha\circ \hat{\alpha}$.

The functoriality of smooth push-forward for the composition $\hat{p}_Z=p_Z\circ q_Z$  yields the identity 
\begin{equation}\label{eqn:BundleInv(A)}\tag{A}
\beta\circ\vartheta_{\hat{i}_Z,\hat{p}_Z, \hat{F}_\bullet}=\vartheta_{i_Z, p_Z, \tilde{F}_\bullet}.
\end{equation}

Let  $p_{\hat{F}^1}:\hat{F}^1\to \hat{Z}$, $p_{\tilde{F}^1}:\tilde{F}^1\to \tilde{Z}$ be the projections.  The smooth push-forward map for the smooth map $q_F$ combined with the exact sequence of vector bundles on $\hat{F}^1$
\begin{equation}\label{eqn:ExactSeqBundleInv}
0\to q_F^*V\to p_{\hat{F}^1}^*i_{\hat{Z}}^*T_{\hat{M}}\to q_F^*p_{\tilde{F}^1}^*i_{\tilde{Z}}^*T_{\tilde{M}}\to0
\end{equation}
induces the morphism
\[
q_{F*}: \hat{F}^1(p_{\hat{F}^1}^*i_{\hat{Z}}^*T_{\hat{M}})_\BM\to
\tilde{F}^1(p_{\tilde{F}^1}^*i_{\tilde{Z}}^*T_{\tilde{M}})_\BM, 
\]
giving us the diagram
\begin{equation}\label{eqn:BundleInv(B)}\tag{B}
\xymatrix{
\hat{Z}(i_{\hat{Z}}^*T_{\hat{M}}-\hat{F}^1)_\BM\ar[r]^-\beta\ar[d]_{0_{\hat{F}^1!}}&
 \tilde{Z}(i^*_{\tilde{Z}}T_{\tilde{M}}-\tilde{F}^1)_\BM \ar[d]^{0_{\tilde{F}^1!}}\\
\hat{F}^1(p_{\hat{F}^1}^*i_{\hat{Z}}^*T_{\hat{M}})_\BM
\ar[r]^-{q_{F*}} &
\tilde{F}^1(p_{\tilde{F}^1}^*i_{\tilde{Z}}^*T_{\tilde{M}})_\BM  
}
\end{equation}
This commutes by Remark~\ref{rem:GysinSmoothCartesian}: the Gysin push-forward commutes with smooth push-forward in a cartersian square of vector bundles.

The cartesian diagram \eqref{eqn:DiagBundleInv} gives rise to the cartesian diagram
\[
\xymatrix{
\mfC_{i_{\hat{Z}}}\ar[r]^-{i_{\hat{\phi}}}\ar[d]_{\mfC(q_M)}&\hat{F}^1\ar[d]^{q_F}\\
\mfC_{i_{\tilde{Z}}}\ar[r]_-{i_{\tilde{\phi}}}&\tilde{F}^1
}
\] 
Using the exact sequence \eqref{eqn:ExactSeqBundleInv} again,  smooth push-forward for the smooth morphism $\mfC(q_M)$ gives the map
\[
\mfC(q_M)_*:\mfC_{i_{\hat{Z}}}(\sigma_{i_{\hat{Z}}}^*T_{\hat{M}})_\BM\to
\mfC_{i_{\tilde{Z}}}(\sigma_{i_{\tilde{Z}}}^*T_{\tilde{M}})_\BM
\]

The compatibility of smooth push-forward and proper pull-back in cartesian squares (see Lemma~\ref{lem:BaseChange}) implies that the diagram
\begin{equation}\label{eqn:BundleInv(C)}\tag{C}
\xymatrix{
\hat{F}^1(p_{\hat{F}^1}^*i_{\hat{Z}}^*T_{\hat{M}})_\BM
\ar[r]^-{q_{F*}}\ar[d]_{i_{\hat{\phi}}^*}  &
\tilde{F}^1(p_{\tilde{F}^1}^*i_{\tilde{Z}}^*T_{\tilde{M}})_\BM \ar[d]_{i_{\tilde{\phi}}^*}\\
\mfC_{i_{\hat{Z}}}(\sigma_{i_{\hat{Z}}}^*T_{\hat{M}})_\BM
\ar[r]_-{\mfC(q_M)_*}&
\mfC_{i_{\tilde{Z}}}(\sigma_{i_{\tilde{Z}}}^*T_{\tilde{M}})_\BM
}
\end{equation}
commutes. 

Finally, it follows directly from the definitions of the various morphisms involved that the diagram
\begin{equation}\label{eqn:BundleInv(D)}\tag{D}
\xymatrix{
\Cst_{\hat{Z}}\ar[r]^-{\Cst(q_M)}\ar[d]_{\alpha_{i_{\hat{Z}}}}&\Cst_{\tilde{Z}}\ar[d]^{\alpha_{i_{\tilde{Z}}}}\\
\mfC_{i_{\hat{Z}}}(\sigma_{i_{\hat{Z}}}^*T_{\hat{M}})_\BM
\ar[r]_-{\mfC(q_M)_*} &
\mfC_{i_{\tilde{Z}}}(\sigma_{i_{\tilde{Z}}}^*T_{\tilde{M}})_\BM
}
\end{equation}
commutes. Putting together the identity \eqref{eqn:BundleInv(A)} with the commutativity of the   diagrams \eqref{eqn:BundleInv(B)}, \eqref{eqn:BundleInv(C)} and \eqref{eqn:BundleInv(D)} gives the identity
\[
s_{i_{\tilde{\phi}}}\circ\vartheta_{i_Z,p_Z, \tilde{F}_\bullet}=\Cst(q_M)\circ s_{i_{\hat{\phi}}}\circ \vartheta_{i_Z, \hat{p}_Z, \hat{F}_\bullet};
\]
composing on the left with $\Cst(p_M)$ and using the functoriality
\[
\Cst(p_M)\circ \Cst(q_M)=\Cst(\hat{p}_M)
\]
completes the proof.
\end{proof}

\begin{proposition} \label{prop:VirtFundClass} Suppose $B$ is affine, take $Z\in \Sch^G/B$  and let  $[\phi]:E_\bullet\to\L_{Z/B}$ be a perfect obstruction theory on $Z$. Choose a closed immersion $i_Z:Z\to M$ with $M\in \Sm^G/B$. Choose a Jouanolou cover $p_M:\tilde{M}\to M$ and let $p_Z:\tilde{Z}\to Z$ be the pull-back $\tilde{M}\times_MZ$. Choose 
 a reduced normalized obstruction theory $\tilde{\phi}:\tilde{F}_\bullet\to (\sI_{\tilde{Z}}/\sI_{\tilde{Z}}^2\to i^*\Omega_{\tilde{M}/B})$ representing $p_Z^![\phi]$.  This gives us the isomorphisms 
$\Cst(p_M)$ \eqref{eqn:StableConeIso} and  $\vartheta_{i_Z, p_Z, \tilde{F}_\bullet}$ \eqref{eqn:JouanolouComp}, and the map $s_{\tilde\phi}$ \eqref{eqn:CompMapDef}.  Composing these morphisms gives
\[
\Phi_{E_\bullet,[\phi]}:=\Cst(p_M)\circ s_{\tilde\phi}\circ \vartheta_{i_Z, p_Z, \tilde{F}_\bullet}:
Z(\V(E_\bullet))_\BM\to \Cst_Z.
\]
Then the  morphism $\Phi_{E_\bullet,[\phi]}$ depends only on the perfect obstruction theory $(E_\bullet,[\phi])$, that is, it is independent of the choice of closed immersion $i_Z$, the choice of Jouanolou cover $p_M$,  the choice of induced obstruction theory $p_Z^![\phi]$ and the choice of reduced normalized obstruction theory $\tilde{\phi}$ representing 
 $p_Z^![\phi]$.
 
Moreover, we have the identity
\[
\vartheta_{i_Z, p_Z, \tilde{F}_\bullet}^*([\tilde{Z}, \tilde{\phi}, i_{\tilde{Z}}]^\vir)=
[\Cst]\circ \Phi_{E_\bullet,[\phi]}.
\]
\end{proposition}

Relying on Proposition~\ref{prop:VirtFundClass}, we may make the following definition.
\begin{definition}\label{def:VirtFundClass} Suppose $B$ is affine, take $Z\in \Sch^G/B$  and let  $[\phi]:E_\bullet\to\L_{Z/B}$ be a perfect obstruction theory on $Z$. We have the fundamental class $[\Cst]\in \mS_B^{0,0}(\Cst)$. Define the {\em virtual fundamental class}
\[
[Z,[\phi]]^\vir\in \mS^\BM_B(Z,\V(E_\bullet))
\]
by
\[
[Z,[\phi]]^\vir=[\Cst]\circ \Phi_{E_\bullet,[\phi]}
\]

If we have a motivic ring spectrum  $\sE$ in $\SH^G(B)$ with unit $\epsilon_{\sE}:\mS_B\to \sE$, we define 
\[
[Z,[\phi]]^\vir_\sE:=\epsilon_{\sE}([Z,[\phi]]^\vir)\in \sE^\BM(Z,\V(E_\bullet)).
\]
\end{definition}

\begin{remark}
It follows from the last assertion in Proposition~\ref{prop:VirtFundClass}, that given a a closed immersion $i_Z:Z\to M$ with $M\in\Sm^G/B$,  a Jouanolou cover  $\tilde{M}\to M$, inducing the Jouanolou cover $p_Z:\tilde{Z}\to Z$, and   a  reduced normalized representative $\tilde\phi:\tilde{F}_\bullet\to L_{\tilde{Z}/B}$ for $p_Z^![\phi]$, we have
\[
[Z,[\phi]]^\vir=\vartheta_{i_Z, p_Z, \tilde{F}_\bullet}^*([\tilde{Z}, \tilde{\phi}, i_{\tilde{Z}}]^\vir).
\]
\end{remark}

\begin{proof}[Proof of Proposition~\ref{prop:VirtFundClass}] 
  Suppose we have fixed a closed immersion $i_Z:Z\to M$ and a Jouanolou cover $\tilde{M}\to M$. By Lemma~\ref{lem:Discussion}, there exists    an induced obstruction theory $p_Z^![\phi]$ on $\tilde{Z}$, which is unique up to canonical isomorphism.

Suppose we have two  reduced normalized obstruction theories $(\tilde{F}_\bullet,\phi)$, $(\tilde{F}'_\bullet,\phi')$ representing $p_Z^![\phi]$. By lemma~\ref{lem:RedNormIso}, there is an isomorphism $\rho_\bullet:\tilde{F}_\bullet\to \tilde{F}'_\bullet$ of perfect obstruction theories. In particular, the map $\rho^1:F^{\prime 1}\to F^1$ satisfies $\rho^1\circ i_{\phi'}=i_\phi$. This gives us the commutative diagram
\[
\xymatrix{
Z(\V(E_\bullet)_\BM\ar[r]^-{\vartheta_{i_Z, p_Z, \tilde{F}_\bullet}}
\ar[d]_{\vartheta_{i_Z, p_Z, \tilde{F}'_\bullet}}&\tilde{Z}(i_{\tilde{Z}}^*T_{\tilde{M}}-\tilde{F}^1)_\BM\ar[d]^{s_{i_\phi}}\\
\tilde{Z}(i_{\tilde{Z}}^*T_{\tilde{M}}-\tilde{F}^{\prime1})_\BM\ar[r]_-{s_{i_{\phi'}}}\ar[ur]_{\gamma(\rho^1)}&\Cst_Z
}
\]
where $\gamma(\rho^1)$ is the isomorphism induced by $\rho^1$.
This yields the independence of $\Phi_{E_\bullet,[\phi]}$ on the choice of reduced normalized obstruction theory representing $p_Z^![\phi]$.  

To show the independence on the choice of Jouanolou cover $p_M:\tilde{M}\to M$ over a fixed closed immersion $i:Z\to M$ we may assume that we are comparing one cover  $p_M:\tilde{M}\to M$ with a second cover $\hat{p}_M:\hat{M}\to M$ which factors as 
\[
\hat{M}\xrightarrow{q_M}\tilde{M}\xrightarrow{p_M}M
\]
with $q_M:\hat{M}\to \tilde{M}$ a vector bundle over $\tilde{M}$. The independence here follows from Lemma~\ref{lem:BundleInv} and what we have already shown.

Finally, suppose we have a smooth morphism $q:N\to M$ and a closed immersion $i'_Z:Z\to N$ with $q\circ i'_Z=i_Z$. By what we have already shown, we may suppose that  $M$ is affine; if we take a Jouanolou cover $\tilde{N}\to N$, then as $Z$ is affine, the cover admits a section over $Z$, so we may replace $N$ with $\tilde{N}$, change notation and assume that $N$ is also affine. With what we have already proven, we may chose a reduced normalized representative $\phi:F_\bullet\to (\sI_{i(Z)}/\sI_{i(Z)}^2\to i_Z^*\Omega_{M/B})$ for $[\phi]$.  

We have the commutative diagram with exact rows
\[
\xymatrix{
0\ar[r]&q_Z^*(\sI_{i(Z)}/\sI_{i(Z)}^2)\ar[r]\ar[d]^d&\sI_{i'(Z)}/\sI_{i'(Z)}^2\ar[r]\ar[d]^d&i^{\prime*}\Omega_{N/M}\ar[r]\ar@{=}[d]&0\\
0\ar[r]&i^{\prime*}q^*\Omega_{M/B}\ar[r]&i^{\prime*}\Omega_{N/B}\ar[r]&i^{\prime*}\Omega_{N/M}\ar[r]&0
}
\]
Choosing a splitting $\rho:i^{\prime*}\Omega_{N/M}\to \sI_{i'(Z)}/\sI_{i'(Z)}^2$ gives us the splitting $\sigma:=d\circ\rho:i^{\prime*}\Omega_{N/M}\to i^{\prime*}\Omega_{N/B}$. This gives the reduced normalized representative for $[\phi]$ (with respect to $i'_Z$)
\[
F'_\bullet:=(F_1\oplus i^{\prime*}\Omega_{N/M}\xrightarrow{q^*d_F\oplus \sigma}
 i^{\prime*}\Omega_{N/B})\xrightarrow{(q^*\phi_1+\rho, \id)}
( \sI_{i'(Z)}/\sI^2_{i'(Z)}\xrightarrow{d}i^{\prime*}\Omega_{N/B}).
 \]

We have the cartesian diagram
\[
\xymatrix{
\mfC_{i'_Z}\ar[d]_{\mfC(q)}\ar[r]^-{i_{\phi'}}&F^1\oplus i^{\prime*}T_{N/M}\ar[d]^{p_1}\\
\mfC_{i_Z}\ar[r]_-{i_\phi}&F^1
}
\]
identifying $\mfC_{i'_Z}$ with the bundle $\sigma_{i'}^*T_{N/M}=\mfC_{i_Z}\times_Z i^{\prime*}T_{N/M}$
over $\mfC_{i_Z}$ and $i_{\phi'}$ with $i_\phi\times \id$. 

We have the isomorphism (Lemma~\ref{lem:CanonIso})
\[
\psi_q:=\psi_{i_Z',i_Z}: \mfC_{i_Z'}(\sigma_{i_Z'}^*T_{N})_\BM\to \mfC_{i_Z}(\sigma_{i_Z}^*T_{M})_\BM.
\]
By Theorem~\ref{thm:IntStableNormalCone}, the diagram
\begin{equation}\label{eqn:VFCDiag1}
\xymatrix{
\Cst_Z\ar[r]^-{\alpha_i}_-\sim\ar[d]_{\alpha_{i'}}^\wr& \mfC_{i_Z}(\sigma_{i_Z}^*T_{M})_\BM \\
\mfC_{i_Z'}(\sigma_{i_Z'}^*T_{N})_\BM \ar[ur]_{\psi_q}
 }
\end{equation}
commutes. 

We have the isomorphism  $\theta_q$  
\[
\Sigma^{\sigma_{i'_Z}^*T_{N}}\cong
\Sigma^{\sigma_{i'_Z}^*T_{N/M}}\circ\Sigma^{\sigma_{i'_Z}^*q^*T_{M}}
\]
induced by the exact sequence
\begin{equation}\label{eqn:ExactSeqFundClassInv}
0\to T_{N/M}\to T_{N}\to q^*T_{M}\to0
\end{equation}
and the isomorphism $nat$  
\[
\Sigma^{\sigma_{i'_Z}^*q^*T_{M}}\circ \mfC(q)^*=\Sigma^{ \mfC(q)^*\sigma_{i_Z}^*T_{M}}\circ \mfC(q)^*\cong
\mfC(q)^*\circ\Sigma^{\sigma_{i_Z}^*T_{M}},
\]
giving the composition
\begin{multline*}
\pi_{\mfC_{i'_Z}!}\circ\Sigma^{\sigma_{i'_Z}^*T_{N}}\circ \mfC(q)^*
 \xrightarrow{\theta_q}
 \pi_{\mfC_{i'_Z}!}\circ\Sigma^{\sigma_{i'_Z}^*T_{N/M}}\circ\Sigma^{\sigma_{i'_Z}^*q^*T_{M}}\circ \mfC(q)^*\\
 \xrightarrow{nat}
  \pi_{\mfC_{i'_Z}!}\circ\Sigma^{\sigma_{i'_Z}^*T_{N/M}}\circ\mfC(q)^*\circ\Sigma^{\sigma_{i_Z}^*T_{M}}
  \xrightarrow{\mfC(q)_*}
  \pi_{\mfC_{i_Z}!}\Sigma^{\sigma_{i_Z}^*T_{M}} .
 \end{multline*}
Let $\Theta_q:=nat\circ \theta_q$. By the definition of  $\psi_q$  we have
\[
\psi_q=(\mfC(q)_*\circ \Theta_q)(1_{\mfC_{i_Z}})
\]

We have  the isomorphism
\[
\theta'_q:   \Sigma^{-F^{\prime1}}\circ\Sigma^{i^{\prime*}T_{N}} \xrightarrow{\sim}
 \Sigma^{-F^1}\circ\Sigma^{i^*T_{M}} 
\]
induced by the exact sequence \eqref{eqn:ExactSeqFundClassInv} and the identity $F^{\prime1}\cong F^1\oplus i^{\prime*}T_{N/M}$. Let $p_1:F^{\prime1}\to F^1$ be the projection.

Consider the diagram
\begin{equation}\label{eqn:VFCDiag22}
\xymatrixcolsep{2pt}
\xymatrix{
\pi_{Z!}\circ \Sigma^{-F^{\prime1}}\circ\Sigma^{i^{\prime*}T_{N}}\ar[rr]^-{\theta'_q}_-\sim\ar[d]_{0_{F^{\prime1}!}}&&
\pi_{Z!}\circ \Sigma^{-F^1}\circ\Sigma^{i^*T_{M}} \ar[d]^{0_{F^1!}}\\
 \pi_{F^{\prime1}!}\circ\Sigma^{p_{F^{\prime1}}^*i^{\prime*}_ZT_{N}}\circ
p_{F^{\prime1}}^*\ar[dr]_\sim^-{\Theta_q}&&  \pi_{F^{1}!}\circ\Sigma^{p_{F^{1}}^*i^{*}_ZT_{M}}\circ p_{F^{1}}^*\\
&\kern-80pt{\pi_{F^{\prime1}!}\circ\Sigma^{p_{F^{\prime1}}^*i^{\prime*}_ZT_{N/M}}\circ
p_1^*\circ \Sigma^{p_{F^{1}}^*i^{*}_ZT_{M}}\circ  p_{F^{1}}^*}\kern-50pt\ar[ru]^{p_{1*}}_\sim
}
\end{equation}

As in the proof of Remark~\ref{rem:GysinSmoothCartesian},  one shows that this diagram commutes by using the functoriality of smooth push-forward,  replacing $0_{F^1!}$ and $0_{F^{\prime1}!}$ with their respective inverses $p_{F^1*}$ and $p_{F^{\prime1}*}$.

The commutativity of the left-hand side of the diagram 
\begin{equation}\label{eqn:VFCDiag23}
\xymatrixcolsep{2pt}
\xymatrix{
\pi_{F^{\prime1}!}\circ\Sigma^{p_{F^{\prime1}}^*i^{\prime*}_ZT_{N}}\circ
p_{F^{\prime1}}^*\ar[dr]_\sim^-{\Theta_q}\ar[dd]_{i^*_{\phi'}}&&  \pi_{F^{1}!}\circ\Sigma^{p_{F^{1}}^*i^{*}_ZT_{M}}\circ p_{F^{1}}^*\ar[dd]^{i^*_{\phi}}\\
&\kern-60pt{\pi_{F^{\prime1}!}\circ\Sigma^{p_{F^{\prime1}}^*i^{\prime*}_ZT_{N/M}}\circ
p_1^*\circ \Sigma^{p_{F^{1}}^*i^{*}_ZT_{M}}\circ  p_{F^{1}}^*}\kern-50pt\ar[ru]_\sim^{p_{1*}}
\ar[dd]_{i^*_{\phi'}}\\
\pi_{\mfC_{i'_Z}!}\circ\Sigma^{\sigma_{i^{\prime}_Z}^*T_{N}}\circ
p_{i'_Z}^*\ar[dr]^-{\Theta_q}_-\sim&&
\pi_{\mfC_{i_Z}!}\circ\Sigma^{\sigma_{i_Z}^*T_{M}}\circ p_{i_Z}^*\\
&\kern-50pt\pi_{\mfC_{i_Z}!}\circ\Sigma^{\sigma_{i^{\prime}_Z}^*T_{N/M}}\circ\mfC(q)^*\circ\Sigma^{\sigma_{i^{\prime}_Z}^*T_{M}}\circ
p_{i'_Z}^*\ar[ru]_\sim^{\mfC(q)_*}\kern-50pt
}
\end{equation}
follows from the naturality $\Sigma^{(-)}:D^\perf_{G, iso}(?)\to \Aut(\SH^G(?))$ and that of the right-hand side  by the commutativity of smooth push-forward with proper pull-back,  Lemma~\ref{lem:BaseChange}.

Putting these two diagrams together and evaluating at $1_Z$ gives us the commutative diagram

\begin{equation}\label{eqn:VFCDiag2}
\xymatrixcolsep{40pt}
\xymatrix{
Z(i^{\prime*}T_{N}-F^{\prime1})_\BM\ar[r]^-{\theta'_q(1_Z)}_-\sim\ar[d]_{i_{\phi'}^*\circ 0_{F^{\prime1}!}}& 
Z(i^*T_{M}-F^1)_\BM \ar[d]^{i_{\phi}^*\circ 0_{F^1!}}\\
\mfC_{i'_Z}(\sigma_{i^{\prime}_Z}^*T_{N})_\BM\ar[r]^-\sim_-{\psi_q}&
\mfC_{i_Z}(\sigma_{i_Z}^*T_{M})_\BM.
}
\end{equation}

The functoriality of $\Sigma^{(-)}:D^\perf_{G, iso}(Z)\to \Aut(\SH^G(Z))$
gives us the commutative diagram of isomorphisms
\begin{equation}\label{eqn:VFCDiag3}
\xymatrixcolsep{0pt}
\xymatrix{
&Z(\V(E_\bullet)_\BM\ar[dl]_{\vartheta_{i_Z, \id_M, F_\bullet}\ }\ar[dr]^{\vartheta_{i'_Z, \id_N, F'_\bullet}}\\
Z(i^{\prime*}T_{N}-F^{\prime1})_\BM\ar[rr]^-\sim_-{\theta'_q(1_Z)} && 
Z(i^*T_{M}-F^1)_\BM
}
\end{equation}
Putting  the diagrams \eqref{eqn:VFCDiag1}, \eqref{eqn:VFCDiag2} and \eqref{eqn:VFCDiag3} together, the definition of $s_\phi$ and $s_{\phi'}$  gives the identity
\[
\alpha_{i_Z}\circ s_{\phi'}\circ\vartheta_{i'_Z, \id_N, F'_\bullet}=
\psi_q\circ \alpha_{i'_Z}\circ s_{\phi'}\circ\vartheta_{i'_Z, \id_N, F'_\bullet}=
\alpha_{i_Z}\circ s_{\phi}\circ\vartheta_{i_Z, \id_M, F_\bullet}
\]
or
\[
s_{\phi'}\circ\vartheta_{i'_Z, \id_N, F'_\bullet}=s_{\phi}\circ\vartheta_{i_Z, \id_M, F_\bullet}
\]
As we are taking the trivial Jouanoulou covers, we have $p_M=\id_M$, $p_N=\id_N$, completing the proof.
\end{proof}

We conclude with a result on compatibility of the fundamental class and virtual fundamental class with respect to base-change.

\begin{proposition}\label{prop:basechange}
Let $\pi_Z:Z\to B$, be in $\Sch^G/B$, let $f:B'\to B$ be a morphism in $\Sch/B$ and let $\pi_{Z'}:Z'=Z\times_BB'\to B'$ be the pull-back. Suppose that the cartesian square
\[
\xymatrix{
Z'\ar[d]\ar[r]^{p_1}&Z\ar[d]\\
B'\ar[r]&B
}
\]
is Tor-independent: $\Tor^i_{\sO_B}(\sO_{B'},\sO_Z)=0$ for $i>0$. Then\\[5pt]
1. Suppose we have a closed immersion $i:Z\to M$ in  $\Sch^G/B$ with $M$ in $\Sm^G/B$, let $M'=M\times_BB'$, and let $i_{Z'}:Z'\to M'$ be the induced closed immersion. Then we have natural isomorphisms $\mfC_{i_{Z'}}\cong \mfC_{i_Z}\times_BB'$, $Def(i_{Z'})\cong Def(i)\times_BB'$ over $M'\times\A^1$.\\
2. The isomorphism  $\mfC_{i_{Z'}}\cong \mfC_{i_Z}\times_BB'$ of (1) induces an isomorphism $\Cst_{Z'}\cong f^*\Cst_Z$ sending the fundamental class $[\Cst_{Z'}]\in \mS_{B'}^{0,0}(\Cst_{Z'})$ to 
$f^*[\Cst_Z]$.\\
3. Let $[\phi]:E_\bullet\to \L_{Z/B}$ be a perfect obstruction theory on $Z$. Then the map $[\phi']:p_1^*E_\bullet\to \L_{Z'/B'}$ defined as the composition
\[
p_1^*E_\bullet\xrightarrow{p_1^*[\phi]} p_1^*\L_{Z/B}\xrightarrow{can} \L_{Z'/B'}
\]
is a perfect obstruction theory on $Z'$, we have a canonical isomorphism $f^*(Z(\V(E_\bullet))_\BM)\cong  Z'(\V(p_1^*E_\bullet))_\BM$ and via this isomorphism we have $p_1^*([Z, [\phi]]^\vir)=[Z'[\phi']]^\vir$.
\end{proposition}

\begin{proof} The Tor-independence implies that $p_1^*\L_{Z/B}\cong \L_{Z'/B'}$ and that $p_1^*[\phi]:p_1^*E_\bullet\to  \L_{Z'/B'}$ is a perfect obstruction theory on $Z'$.

 Letting $p_M:M':=M\times_BB'\to M$ be the projection, the Tor-independence implies that the canonical map $p_M^*\sI_Z\to \sI_{Z'}$ is an isomorphism. This readily implies (1) and gives us the isomorphism of the diagram
\[
\xymatrix{
\mfC_{i_{Z'}}\ar[r]\ar[d]&Def(i_{Z'})\ar[d]&\ar[l]M'\times(\A^1\setminus\{0\})\ar@{=}[d]\\
Z'\ar[r]&M'\times\A^1&\ar[l] M'\times (\A^1\setminus\{0\})
}
\]
with 
\[
\left[
\raise15pt\vbox to 20pt{\hbox{$
\xymatrix{
\mfC_{i_{Z}}\ar[r]\ar[d]&Def(i_{Z})\ar[d]&\ar[l]M\times(\A^1\setminus\{0\})\ar@{=}[d]\\
Z\ar[r]&M\times\A^1&\ar[l] M\times (\A^1\setminus\{0\})
}$}}\right]\times_BB'
\]
(2) follows from this, the naturality of $\Sigma^{-}:\sV(?)\to \SH^G(?)$ and  the localization distinguished triangle with respect to $f^*$, and the base-change isomorphism $Ex(\Delta^*_!)$. This also gives us the isomorphism $f^*(Z(\V(E_\bullet))_\BM\cong Z'(\V(p_1^*E_\bullet))_\BM$. Using the definition of the virtual fundamental class, (3) follows from (1) and (2) together with the compatibility of proper pull-back, smooth push-forward and Gysin push-forward  with base-change by $f^*$ (Remark~\ref{rem:basechange}).
\end{proof}
\section{Comparisons and examples}\label{sec:Comp}
We relate our constructions to the  constructions of \cite{BF} in case of motivic cohomology/Chow groups, or more generally the case of an oriented theory, for example, $K$-theory or algebraic cobordism. For simplicity, we take $G=\{\id\}$ and we remind the reader that we are assuming that $B$ is affine.

\subsection{Oriented theories} Let $\sE$ be an oriented commutative ring spectrum in $\SH(B)$. For $Z\in \Sch^G/B$ and $v\in D^\perf_G(Z)$ of virtual rank $r$, we have the {\em Thom isomorphism}
\[
\vartheta_v:  \Sigma^{v}\pi_Z^!\sE\cong \Sigma^{2r,r}\pi_Z^!\sE
\]
giving the isomorphism on $\sE$-Borel-Moore homology
\[
\vartheta_v^*: \sE^\BM_{a,b}(Z, v)\xrightarrow{\sim}\sE^\BM_{2r+a, r+b}(Z)
\]
If $p_Z:Z\to B$ is smooth of relative dimension $d_Z$, the purity isomorphism $p_{Z\#}\circ\Sigma^{-T_Z}\cong p_{Z!}$ gives the purity isomorphism 
\[
 \sE^\BM_{a,b}(Z, v)\cong 
 (\Sigma^{T_X-v}\sE)^{-a, -b}(Z)\cong \sE^{2d_Z-2r-a, d_Z-r-b}(Z).
 \]

For $i:Z\to M$ a closed immersion in $\Sch^G/B$,  in a smooth dimension $d_M$ $B$-scheme $M$, we thus have
\[
\sE^{a,b}(\Cst_Z)\xymatrix{\ar[r]^{(\alpha_i^{-1})^*}_\sim&}
\sE^{a,b}(\mfC_i(\sigma_i^*T_M)_\BM)\cong
\sE_{2d_M-a, d_M-b}^\BM(\mfC_{i}),
\]

Noting that $\mfC_{i}$  has pure dimension $d_\mfC=d_M$ over $B$, the fundamental class $[\Cst_Z]_\sE$ is thus an element of $\sE^{0,0}(\Cst_Z)\cong\sE^\BM_{2d_\mfC,d_\mfC}(\mfC_{i})$ and the virtual fundamental class $[Z,[\phi]]^\vir_\sE$ associated to a perfect obstruction theory $(E_\bullet, [\phi])$ of virtual rank $r$ lives in $\sE^\BM(Z,\V(E_\bullet))\cong \sE^\BM_{2r,r}(Z)$. 

\subsection{Fundamental classes} Let $\sE$ be an oriented theory. Suppose we have an integral $B$-scheme $D$ and a principal effective Cartier divisor $\mfC$ on $D$ such that $D\setminus \mfC$ is smooth over $B$. Suppose  $D$ has pure dimension $d_\mfC+1$ over $B$, so $\mfC$ has pure dimension $d_\mfC$ over $B$. Let $t\in \Gamma(D, \sO_D)$ be a generator for $\sI_\mfC$. The map $t:D\setminus \mfC\to \G_m$ determines an element 
\[
[t]\in \sE^{1,1}(D\setminus \mfC)\cong 
\sE^\BM_{2d_\mfC+1, d_\mfC}(D\setminus \mfC).
\]
We have the localization sequence
\[
\ldots\to\sE^\BM_{2d_\mfC+1, d_\mfC}(D\setminus \mfC)\xrightarrow{\del}
\sE^\BM_{2d_\mfC, d_\mfC}(\mfC)\xrightarrow{i_{\mfC*}} \sE^\BM_{2d_\mfC, d_\mfC}(D)\to
\ldots
\]
and we have  the fundamental class $[\mfC]_\sE\in \sE^\BM_{2d_\mfC, d_\mfC}(\mfC)$ defined by $[\mfC]_\sE:=\del[t]$.  This class is independent of the choice of defining equation $t$. In case $D=Def(i)$ for our closed immersion $i:Z\to M$, and $t:Def(i)\to \A^1$ is the structure morphism 
$Def(i)\xrightarrow{p}M\times \A^1\xrightarrow{p_2}\A^1$, we have $\mfC=\mfC_i$ and $[\mfC]$ agrees with the fundamental class  $[\Cst_Z]_\sE\in \sE^{0,0}(\Cst_Z)$ (see \S\ref{sec:FundClass}) after identifying $\sE^\BM_{2d_\mfC, d_\mfC}(\mfC)$ with $\sE^{0,0}(\Cst_Z)$ as above. The construction of the fundamental class of $\mfC_{i}$  by this method in the case of algebraic cobordism was described to us by Parker Lowrey and appears in his paper with Timo Sch\"urg \cite{LowreySchuerg}.

\begin{ex} \label{ex:ClassicalExamples} Take $B=\Spec k$,  and  $Z\in\Sch/B$. For  $\sE=H\Z$,  the ring spectrum representing motivic cohomology, $H\Z^\BM_{2d,d}(Z)$ is the classical Chow group $\CH_d(Z)$. For $\sE=\KGL$, the ring spectrum representing Quillen $K$-theory, $\KGL^\BM_{2d,d}(Z)=G_0(Z)$, the Grothendieck group of coherent sheaves on $Z$. For $\sE=\MGL$,  the ring spectrum representing Voevodsky's algebraic cobordism \cite{VoevICM}, and $k$ a field of characteristic zero, $\MGL^\BM_{2d,d}=\Omega_d(Z)$, the algebraic cobordism of \cite{LevineMorel}. 

For $\sE=H\Z$, the fundamental class $[\mfC]_{H\Z}\in \sE^\BM_{2d_\mfC, d_\mfC}(\mfC)$ is the cycle class associated to the scheme $\mfC$. For $\sE=\KGL$, $[\mfC]_{\KGL}$ is the class in $G_0(\mfC)$ of the structure sheaf $\sO_\mfC$. For $\sE=\MGL$, $[\mfC]_\MGL$ is the class associated to the pseudo-divisor $\Div[t]$ on $D$, applied to any resolution of singularities $f:\tilde{D}\to D$ 
\[
[\mfC]_\MGL=f_*(\Div_{\tilde{D}}(f^*t)).
\]
\end{ex}

\subsection{Push-forward and intersection with the 0-section} Let $p:Y\to X$ be a projective map in $\Sch/B$ and let $\sE$ be an oriented theory.  We have the push-forward map
\[
p_*:\sE^\BM_{a,b}(Y)\to \sE^\BM_{a,b}(X)
\]
Given a rank $r$ vector bundle $f:V\to Z$ with a section $s:Z\to V$, we have the maps
\[
f^*:\sE^\BM_{a,b}(Z)\to  \sE^\BM_{a,b}(V, T_{V/Z})\cong \sE^\BM_{a+2r,b+r}(V)
\]
\[
s^!: \sE^\BM_{a+2r,b+r}(V)\cong \sE^\BM_{a,b}(V, T_{V/Z})\to \sE^\BM_{a,b}(Z)
\]
with $s^!=(f^*)^{-1}$, by Lemma~\ref{lem:SectionIsos}. $f^*$ is the usual pull-back for the smooth morphism $f$ and $s^!$ is the classical ``intersection with the 0-section'' defined as the inverse of $f^*$.

\begin{remark} Suppose we have a closed immersion $i:Z\hookrightarrow M$ with $M$ smooth of dimension $d_M$ over $B$. The intrinsic normal cone $\mfC_Z$ as defined by Behrend-Fantechi \cite{BF} is the quotient stack $[\mfC_{i}/i^*T_{M/B}]$. They also define the normal sheaf $\sN_{i}:=\Spec_{\sO_Z}\Sym^*\sI_Z/\sI_Z^2$; the surjection
$\Sym^*\sI_Z/\sI_Z^2\to  \oplus_n\sI^n_Z/\sI_Z^{n+1}$ defines the closed immersion $\mfC_{i}\hookrightarrow \sN_{i}$. This induces the closed immersion of quotient stacks $\mfC_Z\hookrightarrow \sN_Z:=[\sN_{i}/i^*T_{M/B}]$.

Suppose we have a perfect obstruction theory $[\phi]$ of virtual rank $r$ on $i:Z\hookrightarrow M$, with global resolution $(F_1\to F_0)\xrightarrow{\phi_\bullet}(\sI_Z/\sI_Z^2\to i^*\Omega_{M/B})$. We may assume that $(F_\bullet, \phi)$ is normalized. The assumption that $\phi$ is a perfect obstruction theory implies that $\phi$ induces closed immersions
\[
\mfC_Z\hookrightarrow \sN_Z\hookrightarrow [F^1/F^0].
\]
Let $\mfC(F_\bullet)\subset \sN(F_\bullet)\subset F^1$ be the pull-back of this sequence of closed immersions by the quotient map $F^1\to [F^1/F^0]$.

One can describe $\sN(F_\bullet)$ explicitly as follows: We have the commutative diagram
\[
\xymatrix{
F_1\ar[r]^{d_F}\ar[d]_{\phi_1}&F_0\ar[d]^{\phi_0}\\
\sI_Z/\sI_Z^2\ar[r]_d&i^*\Omega_{M/B}
}
\]
Let $F:=\sI_Z/\sI_Z^2\times_{i^*\Omega_{M/B}}F_0$. The map $(\phi_1, d_F)$ gives a surjection $\phi_\sN:F_1\to F$;  $\sN(F_\bullet)$ is the closed subscheme $\Spec_{\sO_Z}\Sym^*F$ of $F^1=\Spec_{\sO_Z}\Sym^*F_1$. 

The virtual fundamental class $[Z,[\phi]]^\vir_{BF}$ as defined by Behrend-Fantechi {\it loc. cit.} is the element of $\CH_r(Z)$ given by 
\[
[Z,[\phi]]^\vir_{BF}:=0_{F^1}^!([\mfC(F_\bullet)]).
\]

Suppose that we have a Jouanolou cover $p_M:\tilde{M}\to M$, with pull-back $p_Z:\tilde{Z}\to Z$. An induced perfect obstruction theory $p_Z^![F_\bullet]\to (\sI_{\tilde{Z}}/\sI_{\tilde{Z}}^2\to i_{\tilde{Z}}^*\Omega_{\tilde{M}/B})$ is defined (see Lemma~\ref{lem:Discussion}). 

Writing 
$(\tilde{F}_1\to \tilde{F}_0):=p_Z^![F_\bullet]$, we have $\tilde{F}_1=p_Z^*F_1$ and $\tilde{F}_0=p_Z^*F_0\oplus \Omega_{\tilde{Z}/Z}$. Thus, we have the isomorphism of quotients of $\tilde{F}_1$
\[
\tilde{F}:=\sI_{\tilde{Z}}/\sI^2_{\tilde{Z}}\times_{i_{\tilde{Z}}^*\Omega_{\tilde{M}/B}}\tilde{F}_0\cong p_Z^*F
\]
which shows that $\sN(\tilde{F}_\bullet)\subset \tilde{F}^1$ is equal to $p_Z^*\sN(F_\bullet)$. Thus 
\[
\mfC(\tilde{F}_\bullet)=p_Z^*\mfC(F_\bullet)\subset p_Z^*F^1=\tilde{F}^1,
\]
which implies that
\[
p_Z^*([Z,[\phi]]^\vir_{BF})=[\tilde{Z},p_Z^![\phi]]^\vir_{BF}
\]
in $\CH_{r+d}(\tilde{Z})$, where $d$ is the rank of $\Omega_{\tilde{Z}/Z}$.

Since $p_Z:\tilde{Z}\to Z$ induces an isomorphism
\[
p_Z^*:\CH_r(Z)\to \CH_{r+d}(\tilde{Z}),
\]
we may assume that $Z$ is affine for the purpose of comparing our construction of virtual fundamental classes with that of Behrend-Fantechi.

Assuming then that $Z$ is affine, with a closed immersion $i:Z\to M$ into a smooth affine $B$-scheme $M$, we may take a reduced normalized representative $(F_1\to F_0)\to (\sI_Z/\sI_Z^2\to i^*\Omega_{M/B})$ of a given rank $r$ perfect obstruction theory $[\phi]$; let $r_i:\text{rank}(F_i)$. In this case, we have $\mfC(F_\bullet)=\mfC_{i}\subset F^1$, and $d_\mfC=r_0$. We have already identified our construction of the fundamental class $[\mfC_{i}]_{H\Z}$ with the cycle class $[\mfC_{i}]\in \CH_{r_0}(\mfC_{i})$; we have also identified the push-forward map $i_{\mfC*}:\CH_{r_0}(\mfC_{i})\to \CH_{r_0}(F^1)$ and the intersection with the 0-section $0_{F^1}^!:\CH_{r_0}(F^1)\to \CH_{r_0-r_1}(Z)$, as defined here, with the classical ones. This gives the identity of virtual fundamental classes
\[
[Z,[\phi]]^\vir_{BF}=[Z,[\phi]]^\vir
\]
in $\CH_r(Z)$. 

Of course, the Behrend-Fantechi theory is defined for perfect obstruction theories on Deligne-Mumford stacks, whereas the theory presented here is limited to quasi-projective $B$-schemes for affine $B$.   
\end{remark}

\subsection{Gromov-Witten invariants  for oriented theories} If $[\phi]$ is a virtual rank zero perfect obstruction theory on some $Z\in \Sch/B$, and $\sE$ is an oriented ring spectrum in 
$\SH(B)$, the virtual fundamental class lives in $\sE^\BM_{0,0}(Z)=\sE^{0,0}(\pi_{Z!}(1_Z))$. If $\pi_Z:Z\to B$ is proper, we have the push-forward map in $\sE$-cohomology 
\[
\pi_{Z*}:\sE_{a,b}^\BM(Z)\to \sE^\BM_{a,b}(B)=\sE^{-a,-b}(B),
\]
induced by the map $\pi_Z^*:1_B\to \pi_{Z!}(1_Z)$, giving the GW-invariant
\[
\Deg_\sE([Z,[\phi]]^\vir):=\pi_{Z*}([Z,[\phi]]^\vir)\in \sE^{0,0}(B).
\]
This is the classical ``degree of the virtual fundamental class'' in case $\sE=H\Z$. For more general theories, we may have non-zero invariants for perfect obstruction theories of non-zero ranks which give rise to non-zero degrees:  
\[
\Deg_\sE([Z,[\phi]]^\vir):=\pi_{Z*}([Z,[\phi]]^\vir)\in \sE^\BM_{2r,r}(B) =\sE^{-2r,-r}(B)
\]
for $[\phi]$ of virtual rank $r$.

If we have a morphism $f:Z\to W$, one can twist $[Z,[\phi]]^\vir$ by classes coming from $W$; if $Z$ is proper over $B$, pushing forward gives the descendant classes in $\sE^\BM_{*,*}(B)$.

\subsection{Gromov-Witten invariants  for $\SL$-oriented theories} We now consider theories $\sE$ which are not oriented in the sense of the previous section, but are rather $\SL$-oriented. This means that, given a perfect complex $E_\bullet$  on some $Z\in \Sch/B$ of virtual rank $r$ and virtual determinant $\det E_\bullet$,  there is a canonical  isomorphism 
\[
\lambda_{E_\bullet}:\Sigma^{\V(E_\bullet)-\V(\sO^r_Z)}\pi_Z^*\sE\cong \Sigma^{\V(\det E_\bullet)-\V(\sO_Z)}\pi_Z^*\sE
\]
Moreover, for $L$ a line bundle on $Z$, we have a canonical isomorphism 
\[
\Sigma^{\V(L^{\otimes 2})-\V(\sO_Z)}\pi_Z^*\sE\cong \pi_Z^*\sE.
\]
This gives us the isomorphisms
\[
\sE^\BM_{a,b}(Z, \V(E_\bullet))\cong \sE^\BM_{a+2r-2,b+r-1}(Z, \V(\det(E_\bullet)))
\]
and
\[
\sE^\BM_{a,b}(Z)\cong \sE^\BM_{a-2,b-1}(Z, L^{\otimes 2})
\]

This also implies that, if $\beta:V\to V$ is an automorphism of a vector bundle $V\to Z$, the induced map
\[
\Sigma^\beta: \sE^\BM_{a,b}(Z,V)\to \sE^\BM_{a,b}(Z,V)
\]
is given by  multiplication by the automorphism $\<\det\beta\>\in \End_{\SH(Z)}(1_Z)$.

Suppose we have a rank $r$ perfect obstruction theory $([\phi], E_\bullet)$ on some $Z\in \Sch/B$, with $\pi_Z:Z\to B$ projective over $B$. The virtual fundamental class $[Z,[\phi]]^\vir$ lives in $\sE^\BM(Z, \V(E_\bullet))\cong \sE^\BM_{2r,r}(Z, \V(\det E_\bullet)-\V(\sO_Z))$. Given an isomorphism 
\[
\alpha:\det E_\bullet\to L^{\otimes 2}
\]
for some line bundle $L$ on $Z$, we have the isomorphism 
\[
\lambda_\alpha:\sE^\BM_{2r,r}(Z, \V(\det E_\bullet)-\V(\sO_Z))\xrightarrow{\sim}
\sE^\BM_{2r,r}(Z),
\]
so we may push $\lambda_\alpha([Z,[\phi]^\vir)$ forward by $\pi_Z$ to give
\[
\Deg_\sE([Z,[\phi]^\vir, \alpha):=\pi_{Z*}(\lambda_\alpha([Z,[\phi]^\vir))\in  \sE^\BM_{2r,r}(B).
\]

\begin{ex} We take $B=\Spec k$, $k$ a perfect field, $G=\{\id\}$ and $\sE=H_0(\mS_k)$. We have the sheaf of graded rings on $\Sm_k$ given by {\em Milnor-Witt $K$-theory}, $\sK^{MW}_*$ (see 
\cite{MorelA1} for details).  The 0-th homology $H_0(\mS_k)$ with respect  to the homotopy $t$-structure on $\SH(k)$ represents the theory of Milnor-Witt $K$-theory: for $X\in \Sm/k$, there is a canonical isomorphism
\[
H_0(\mS_k)^{a+b,b}(X)\cong H^a_{Nis}(X, \sK^{MW}_b)
\]

$H_0(\mS_k)$ is $\SL$-oriented, giving the canonical isomorphism
\[
H_0(\mS_k)^\BM_{a,b}(Z, v-\V(\sO_Z^r))\cong H_0(\mS_k)_{a, b}^\BM(Z, \det v-\V(\sO_Z))
\]
for $v$ of virtual rank $r$ on $Z\in \Sch/B$. This gives the isomorphism
\[
H_0(\mS_k)^\BM_{2n,n}(Z, \V(E_\bullet))\cong \widetilde{\CH}_n(Z,\det E_\bullet).
\]
for $Z\in \Sch/k$, $E_\bullet\in D^\perf(Z)$, where $\widetilde{\CH}_*$ is the Chow-Witt theory of Barge-Morel \cite{BargeMorel} and Fasel \cite{FaselCW} (also called the oriented Chow groups). For example, we have $\widetilde{\CH}_0(\Spec k)=\GW(k)$, the Grothendieck-Witt ring of non-degenerate symmetric bilinear forms over $k$.

For $i:Z\hookrightarrow M$, $M$ smooth of dimension $d_M$ over $k$, and $(\phi, E_\bullet)$ a rank $r$ perfect obstruction theory,  this gives us the fundamental class and virtual fundamental class 
\begin{align*}
&[\Cst_Z]\in \widetilde{\CH}_{d_M}(\mfC_{i}, \sigma_i^*\omega_{M/k})\\
&[Z, [\phi]]^\vir\in H_0(\mS_k)^\BM_{2r, r}(Z, \V(\det E_\bullet)-\V(\sO_Z)) =  \widetilde{\CH}_r(Z, \det E_\bullet).
\end{align*}

Thus, if we have a rank 0 perfect obstruction theory $(E_\bullet, [\phi])$ on some $Z\in \Sch/k$, with $Z$ projective over $k$, and an isomorphism $\alpha:\det E_\bullet\xrightarrow{\sim} L^{\otimes 2}$ for some line bundle $L$ on $Z$, we have 
\[
\Deg_{H_0(\mS_k)}([Z,[\phi]]^\vir, \alpha):=\pi_{Z*}(\lambda_\alpha([Z,[\phi]]^\vir))\in
 H_0(\mS_k)^\BM(\Spec k)=K^{MW}_0(k)=\GW(k).
\]

More generally, if $E_\bullet$ has virtual rank $r$, and we have a morphism $f:Z\to W$ with $W\in \Sm/k$, a line bundle $L'$ on $W$, a line bundle $L$ on $Z$ and an isomorphism 
\[
\alpha: \det E_\bullet\otimes f^*L'\to L^{\otimes 2}
\]
then a class  
\[
\beta\in  
H^r(W, \sK^{MW}_{r+s}(L'))=H_0(\mS_k)^{2r+s, r+s}(W, \V(\sO_W)-\V(L'))
\]
gives via the cap product
\begin{multline*}
H_0(\mS_k)^\BM_{2r, r}(Z, \V(\det E_\bullet)-\V(\sO_Z))\times H_0(\mS_k)^{2r+s, r+s}(W, \V(\sO_W)-\V(L'))\\
\xrightarrow{-\cap f^*(-)}
H_0(\mS_k)^\BM_{-s, -s}(Z, \V(\det E_\bullet\otimes f^*L')-\V(\sO_Z))
\end{multline*}
the element
\[
\lambda_\alpha([Z,[\phi]^\vir]\cap f^*\beta)\in  H_0(\mS_k)^\BM_{-s,-s}(Z),
\]
and we can define the descendant class
\begin{multline*}
\Deg^s_{H_0(\mS_k)}([Z,[\phi]^\vir]\cap f^*\beta,\alpha):=
\pi_{Z*}(\lambda_\alpha([Z,[\phi]^\vir]\cap f^*\beta))\\\in 
H_0(\mS_k)^\BM_{-s, -s}(\Spec k)=H_0(\mS_k)^{s, s}(\Spec k)=K^{MW}_s(k).
\end{multline*}

There is a universal $\SL$-oriented theory, $\MSL$, with $\MSL_n$ the Thom space of the universal bundle $\tilde{E}_n\to \BSL_n$. Just as for $\MGL_n$, $\MSL_n^{-2r, -r}(k)$ is non-zero for all $r\ge0$, so we have a non-trivial target for the degree map for perfect obstruction theories of all non-negative ranks, but having a trivialized determinant bundle (up to a square).
\end{ex}

\section{Generating ideals, \hbox{$\A^1$}-local degree and critical loci} \label{sec:LocalDeg}
The most elementary type of obstruction theory is one that is already normalized and reduced. Fix a scheme $B$. For $X\to B$ a $B$-scheme we let $\pi_X:X\to B$ denote the structure morphism. Let $M\to B$ be a smooth $B$-scheme and $i:Z\hookrightarrow M$ a closed subscheme. This gives us the cone $\mfC_{i}$ and the fundamental class $[\mfC_i]\in \mS_B^\BM(\mfC_i,\sigma_i^*T_M)$.

We  suppose we have a locally free sheaf $\sV$ on $M$ and an $\sO_Z$-linear surjective map $F:\sV^\vee\otimes\sO_Z\to \sI_Z/\sI_Z^2$; if $M$ is quasi-projective, this always exists. We let $\del_\phi:\sV^\vee\otimes\sO_Z\to \Omega_{M/B}\otimes\sO_Z$ be the map $d\circ \phi_1$, with $d:\sI_Z/\sI_Z^2\to  \Omega_{M/B}\otimes\sO_Z$ induced by  the  canonical derivation $d:\sO_M\to \Omega_{M/B}$. This gives us the perfect obstruction theory
\[
\phi=(F, \id):(\sV^\vee\otimes\sO_Z\xrightarrow{\del_\phi} \Omega_{M/B}\otimes\sO_Z)\to 
(\sI_Z/\sI_Z^2\xrightarrow{d}  \Omega_{M/B}\otimes\sO_Z)
\]
on $Z$, which is already reduced and normalized.

The map $F$ induces the closed immersion of $Z$-schemes $i_\phi:\mfC_i\hookrightarrow i^*V$, where $V\to M$ is the vector bundle $\V(\sV^\vee)$. The associated virtual fundamental class is then
\[
[Z, [\phi]]^\vir:=0_{i^*V}^!i_{\phi*}([\mfC_i])\in \mS_B^\BM(Z, i^*T_M-i^*V).
\]
Since $\phi$ is determined by $F$, we write this as $[Z, F]^\vir$

A choice of an isomorphism $\psi:T_M\xrightarrow{\sim} V$ (if one exists) simplifies this to a  virtual fundamental class $[Z, F]^\vir\in \mS_B^\BM(Z)$. If we pass to an $\SL$-oriented theory $\sE$ and $V$ and $T_M$ have the same rank, then we only need an isomorphism $\rho: \det\sV\otimes  \omega_{M/B}\otimes\sO_Z\to \sL^{\otimes 2}$ for some invertible sheaf $\sL$ on $Z$ to reduce to   $[Z, F]^\vir_\sE\in \sE^\BM(Z)$. For example, if $\sV=\Omega_{M/B}$, we have the identity $\det \sV\otimes\omega_{M/B}=\omega_{M/B}^{\otimes 2}$. If $M$ has even dimension $2m$ and $\sV=\Omega_{M/B}\otimes\sL$ for some invertible sheaf $\sL$, then we have the canonical isomorphism $\det\sV\otimes  \omega_{M/B}\cong(\omega_{M/B}\otimes\sL^{\otimes m})^{\otimes 2}$.

\begin{ex} Let $R$ be a commutative ring, let $B=\Spec R$,   and let $Z\subset \A^n_B$ be a closed subscheme. Let $I\subset R[X_1, \ldots, X_n]$ be the ideal of $Z$ and choose $F_1,\ldots, F_n\in I$ that generate $I/I^2$ as $\sO_Z$-module.  

We have  $\sO_{Z}$-linear surjection $F:\Omega^\vee_{\A^n_B/B}\otimes\sO_{Z}\to I/I^2$ defined by sending the basis element $\del/\del X_i$ to the image of $F_i$ in $I/I^2$. Using the trivialization of $T_{\A^n_B/B}$ and $T^\vee_{\A^n_B/B}$  by  the bases $\{\del/\del X_i\}_i$ and $\{dX_i\}_i$ we have the canonical isomorphism $\Sigma^{i^*T_{A^n_B/B}-i^*T^\vee_{\A^n_B/B}}\cong \id$, giving us  the virtual fundamental class $[Z, F]^\vir\in \mS_B^\BM(Z)= \mS_B^{0,0}(Z_\BM)$. Note that 
$[Z, F]^\vir$ depends only on the image of the $F_i$ modulo $I^2$.
\end{ex}

\begin{definition}[$\A^1$-local degree \hbox{\cite[Definition 10]{KW}}]\label{def:A1LocDeg} Let $B=\Spec A$ be an affine scheme and let   $g=(g_1,\ldots, g_n):\A^n_B\to \A^n_B$ be a polynomial map, $g_i\in A[X_1,\ldots, X_n]$. Suppose that  $g^{-1}(0_B)_\red$ is a disjoint union, $g^{-1}(0_B)=x\amalg x'$, with $g_{|x}:x\to B$ finite.  Choose $U\subset \A^n$ an open neighborhood of $x$ with $U\cap g^{-1}(0_B)=x$. The {\em $\A^1$-local  degree} of $g$ along $x$, $\delta_{\A^1}(g, x)$, is the element of $\mS_B^{0,0}(B)=\End_{\SH(B)}(\mS_B)$ given by stabilizing the composition
\[
S^{2n,n}_B=\P^n_B/\P^{n-1}_B\to \P^n_B/\P^n_B\setminus x \xleftarrow{\sim}U/U\setminus  x\xrightarrow{g_{|(U,U\setminus x)}} \P^n_B/\P^n_B\setminus\{0_B\}\xleftarrow{\sim} \P^n_B/\P^{n-1}_B=S^{2n,n}_B.
\]
The map $U/U\setminus x\to \P^n_B/\P^n_B\setminus x$ is an isomorphism by excision and the map  $\P^n_B/\P^{n-1}_B\to  \P^n_B/\P^n_B\setminus\{0_B\}$ is an isomorphism by homotopy invariance. It is easy to see that this composition is independent of the choice of $U$.
\end{definition}

\begin{remarks}\label{rem:LocalDegAdditivity} Let $g:\A^n_B\to \A^n_B$ and $x\subset g^{-1}(0_B)\subset \A^n_B$ be as in Definition~\ref{def:A1LocDeg}. \\[5pt]
1. Suppose $x$ is a disjoint union $x=x_1\amalg x_2$.   Then 
\[
\delta_{\A^1}(g, x)=\delta_{\A^1}(g,x_1)+\delta_{\A^1}(g, x_2).
\]
This follows by considering the Nisnevich cover 
\[
U\setminus x_2\amalg U\setminus x_1\to U, 
\]
where $U\subset \A^n$ is an open subscheme  with $g^{-1}(0_B)\cap U= x$.
\\[3pt]
2. Letting $f:B'\to B$ be a morphism, we have $f^*(\delta_{\A^1}(g, x))\in \mS_B^{0,0}(B')$, we have the pull-back morphism $g':\A^n_{B'}\to \A^n_{B'}$ and $x':=f^{-1}(x)\subset g^{\prime-1}(0_{B'})$. Then
\[
f^*(\delta_{\A^1}(g, x))=\delta_{\A^1}(g', x').
\]
Indeed, if we take $U\subset \A^n_B$ with $g^{-1}(0_B)_\red\cap U=x$, then letting $U'=U\times_BB'$, we have $g^{\prime-1}(0_{B'})_\red\cap U'=x'$, and the result follows by applying the base-change isomorphism $Ex(\Delta^*_{\#})$ to the sequence of maps defining $\delta_{\A^1}(g, x)$ and $\delta_{\A^1}(g', x')$.\\[2pt]
3. In case $B=\Spec k$, $k$ a perfect field, Morel's theorem identifies $\mS_k^{0,0}(\Spec k)$ with $\GW(k)$, so we have $\delta_{\A^1}(g, x)\in \GW(k)$. 
\end{remarks}

\begin{remark}\label{rem:Automorphism} Let $K$ be a perfect field, $f:\A^n_K\to \A^n_K$ a linear automorphism and $p:\A^n_K\to \Spec K$ the projection. Then the map $f_*:p_!(1_{\A^n})\to p_!(1_{\A^n})$ is multiplication by the rank one quadratic form $\<\det f\>$. Indeed, we may use a matrix representation for $f$. Since $\SL_n(K)$ is generated by elementary matrices, $f$ is $\A^1$-homotopic to the map $(x_1,\ldots, x_n)\mapsto (ux_1,x_2,\ldots, x_n)$ with $u=\det f$.  

Via the canonical isomorphism $p_!=p_\#\circ\Sigma^{-T_{\A^n/K}}$, we have $p_!(1_{\A^n})\cong \Sigma^n_{\P^1}(1_{\Spec K})$, with the action of $f$ going over to $\Sigma^{n-1}_{\P^1}(p_{\A^1!}(\det f))$, with $\det f:\A^1_K\to \A^1_K$ the multiplication map. This reduces us to the case $n=1$. In this case, via the isomorphism $p_!(1_{\A^1})\cong \Sigma_{\P^1}(1_{\Spec K})$, multiplication by $u=\det f$ goes over to the map $\P^1\to \P^1$ sending $[x_0:x_1]$ to  $[x_0:ux_1]$. Morel's isomorphism $\GW(K)\to \End_{\SH(K)}(1_{\Spec K})$ sends the quadratic form $\<u\>$ to the stable version of this latter map.
\end{remark}

\begin{lemma}\label{lem:VIrDegEt} Let $k$ be a perfect field, let $g:\A^n_k\to \A^n_k$ be a polynomial map, $g=(g_1,\ldots, g_n)$ and suppose we have closed points $x_1,\ldots, x_r$ of $\A^n_k$ which are all isolated points of $g^{-1}(0)$.  Let $Z= \{x_1,\ldots, x_r\}$ and suppose that the map $g$ is \'etale in a neighborhood of $Z$. Let $\pi_Z:Z\to \Spec k$ be the projection, giving the push-forward map  $\pi_{Z*}:\mS_k^\BM(Z)\to \mS_k^\BM(\Spec k)=\mS_k^{0,0}(\Spec k)$. Then  
\[
\pi_{Z*}([Z,g]^\vir)=\delta_{\A^1}(g, Z)
\]
\end{lemma}

\begin{proof}  We use the standard basis $dX_1,\ldots, dX_n$ for $\Omega_{\A^n/k}$. This gives us the canonical isomorphisms
\[
\mfC_{Z\subset\A^n}\cong \A^n_Z,\ T_{\A^n/k}^\vee\otimes\sO_Z\cong \A^n_Z
\]
via which the isomorphism $i_g:\mfC_{Z\subset\A^n}\to T_{\A^n/k}^\vee\otimes\sO_Z$ becomes the $\sO_Z$-linear map $\A^n_Z\to \A^n_Z$ with matrix the Jacobian matrix $\sJ(g)=(\del g_i/\del X_j)$ restricted to $Z$. Letting $J(g)=\det\sJ(g)$, it follows from Remark~\ref{rem:Automorphism} and the definition of $[Z,g]^\vir$ that $[Z,g]^\vir$ is the rank one quadratic form $\<J(g)(Z)\>\in \GW(Z)=\mS_k^{0,0}(Z)$, that is, at $x_i\in Z$, 
$[Z,g]^\vir$ takes the values $\<J(g)(x_i)\>\in \GW(k(x_i)\>=\mS_k^{0,0}(x_i)$. 

Since $Z\to \Spec k$ is \'etale, we have $Z_\BM=Z$ and the map $\pi_{Z*}:\mS_k^{0,0}(Z)\to \mS_k^{0,0}(\Spec k)$ is identified with the trace map 
\[
\Tr_{k[Z]/k}=\sum_i\Tr_{k(x_i)/k}:\prod_{i=1}^r\GW(k(x_i))\to \GW(k)
\]
(see \cite[Lemma 5.3]{HoyoisTrace}).
By \cite[Proposition 14]{KW}, $\delta_{\A^1}(g, Z)=\Tr_{k[Z]/k}\<J(g)(Z)\>$, thus $\pi_{Z*}([Z,g]^\vir)=\delta_{\A^1}(g, Z)$, as claimed.
\end{proof}

We can remove the condition that $g$ is \'etale along $Z$ by a deformation argument, taken from \cite{KW}.

\begin{proposition}\label{prop:VirDeg} Let $k$ be  perfect field of characteristic different from 2. Let $g=(g_1,\ldots, g_n):\A^n_k\to\A^n_k$ be a polynomial map. Suppose that $g^{-1}(0)$ is a disjoint union of closed subschemes $Z\amalg Z'$ with $Z$ of pure dimension zero and write $Z_\red=z$.   Then $\pi_{Z*}([Z, g]^\vir)\in \mS_k^{0,0}(k)=\GW(k)$ is the $\A^1$ local degree $\delta_{\A^1}(g, z)$.
\end{proposition}

\begin{proof} Since the map $\GW(k)\to \GW(k')$ is injective for $k\subset k'$ a field extension that is the  union over a tower of finite extensions $k\subset k_\alpha$ of odd degree, we may assume that $k$ is an infinite field.

Let $\mathfrak{m}\subset k[X_1,\ldots, X_r]$ be the ideal of $z$ Then $Z$ is a complete intersection component of the subscheme of $\A^n_k$ defined by $(g_1,\ldots, g_n)$. Moreover, the cone $\mfC_{Z\subset\A^n}$, the fundamental class $[\mfC_{Z\subset\A^n}]$ and the map $\phi_{g}$ are unchanged if we replace the $g_i$ with polynomials $g_i'\in k[X_1,\ldots, X_n]$ such that $g_i'-g_i\in \mathfrak{m}^b$ for $b>>0$. 
By \cite[Lemma 15(3)]{KW}, the same holds for the $\A^1$ local degree $\delta_{\A^1}(g, z)$.

Adding to each $g_i$ a suitably general $h_i\in (X_1,\ldots, X_n)^b$ for sufficiently high $b$, we may assume that each of the $g_i$ have the same degree $d$ and that the map $g$ extends to a  morphism $\bar{g}:\P^n_k\to \P^n_k$ satisfying
\begin{equation}\label{eqn:Assumption18}
\vbox{
\begin{enumerate}
\item[(1)] $\bar{g}$ is finite, flat and of degree prime to the characteristic of $k$.
\item[(2)] $\bar{g}$ is \'etale at each point of $F^{-1}(0)\setminus Z$
\item[(3)] $\bar{g}^{-1}(\A^n)\subset \A^n$
\end{enumerate}
}
\end{equation}
This is proven in \cite[Lemmas 19-21, Proposition 22]{KW} under the assumption that $Z$ is supported at 0, but the same proof works in the more general case.  We construct a morphism over $\A^1$, $\sG:\P^n_{\A^1}\to \P^n_{\A^1}$,
satisfying \eqref{eqn:Assumption18}(1) and in addition
\begin{enumerate}
\item[(2$'$)] There is an open subset $V\subset \A^1$ containing $1$ such that $\sG$ is \'etale over a neighborhood of $\sG^{-1}(0\times V)$.
\item[(3$'$)] $\sG^{-1}(\A^n\times \A^1)\subset \A^n\times\A^1$
\item[(4$'$)]  Letting $\sG_\lambda:\P^n_k\to \P^n_k$ be the pull-back of $\sG$ over $\lambda\in \A^1(k)$, we have $\sG_0=\bar{g}$.
\end{enumerate}
Indeed, since $\bar{g}$ is finite and of degree prime to the characteristic, $\bar{g}$ is \'etale over a dense open subset $U\subset \P^n_k$. Since $k$ is infinite, there is a $k$-point $u\in U$, and thus $\bar{g}$ is \'etale over a neighborhood of   $\bar{g}^{-1}(u)$. Let $\phi:\A^1\to \Aut\P^n_k$ be the morphism sending $t$ to translation by $tu$ (in the Euclidean subgroup of affine linear automorphisms of $\A^n$, embedded as a subgroup of $\Aut(\P^n)$ in the usual way), and define $\sG$ by $\sG(x,t)=\phi(-t)(\bar{g}(x))$.

Let $Z\subset \A^n\times\A^1$ be the closed subscheme $\sG^{-1}(0\times\A^1)$. Then $\pi_Z:Z\to \A^1$ is finite and flat with fiber over 0  a disjoint union of two closed subschemes $Z\amalg Z'$. Moreover, $Z'$ is reduced and the map $g$ is \'etale on a neighborhood of $Z'$. Let $\tilde{\sG}=\sG_{|\A^n\times\A^1}:\A^n\times\A^1\to \A^n\times\A^1$, and $\tilde{\sG}_\lambda$ the restriction of $\tilde{\sG}$ over $\lambda\in \A^1(k)$, so $\tilde{\sG}_0=g$. Similarly, let $Z_\lambda$ be the fiber of $Z$ over $\lambda$.

Replacing $\SH(k)$ with $\SH(\A^1_k)$, we have the virtual fundamental class 
\[
[Z, \tilde{\sG}]^\vir\in \Hom_{\SH(\A^1_k)}(\pi_{Z!}(1_Z), \mS_{\A^1})
\]
its push-forward $\pi_{Z*}([Z, \tilde{\sG}]^\vir)\in \End_{\SH(\A^1_k)}(\mS_{\A^1})$ and the $\A^1$ local degree of $\tilde{\sG}$, $\delta_{\A^1}(\tilde{\sG}, Z_\red)\in \End_{\SH(\A^1_k)}(\mS_{\A^1})$. Since $Z$ is flat over $\A^1$, we have
 for each $\lambda\in \A^1(k)$, with inclusion $i_\lambda:\Spec k\to \A^1$,  the identity
\[
i_\lambda^*(\pi_{Z*}([Z, \tilde{\sG}]^\vir))=
\pi_{Z_\lambda*}([Z, \tilde{\sG}_\lambda^\vir])
\]
(see Proposition~\ref{prop:basechange}). Similarly,
\[ 
i_\lambda^*(\delta_{\A^1}(\tilde{\sG}, Z_\red))=\delta_{\A^1}( \tilde{\sG}_\lambda, Z_{\lambda\red}).
\]
On the other hand, for $p:\A^1_k\to \Spec k$ the projection, we have
\[
\End_{\SH(\A^1_k)}(\mS_{\A^1})=\Hom_{\SH(\A^1)}(1_{\A^1}, p^*(1_k))=
 \Hom_{\SH(k)}(p_\#1_{\A^1}, 1_k)=\mS^{0,0}_k(\A^1_k),
\]
so by homotopy invariance and Remark~\ref{rem:LocalDegAdditivity} we have 
\begin{align*}
&\pi_{Z_1*}([Z, \tilde{\sG}_1]^\vir)=\pi_{Z_0*}([Z, \tilde{\sG}_0]^\vir)=\pi_{Z*}([Z, g]^\vir)+\pi_{Z'*}([Z', g]^\vir)\\
&\delta_{\A^1}(\tilde{\sG}_1, Z_{1\red})=\delta^{\A^1}(\tilde{\sG}_0, Z_{0\red})=
\delta_{\A^1}(g_0, Z_\red)+\delta_{\A^1}(g, Z')
\end{align*}
By Lemma~\ref{lem:VIrDegEt}, we have
\[
\pi_{Z_1*}([Z, \tilde{\sG}_1]^\vir)=\delta_{\A^1}(Z_1, \tilde{\sG}_1)\text{ and }\pi_{Z'*}([Z', g]^\vir)=\delta_{\A^1}(g, Z'),
\]
so 
\[
\pi_{Z*}([Z, g]^\vir)=\delta_{\A^1}(g, Z_\red).
\]
\end{proof} 

\begin{remark} Proposition~\ref{prop:VirDeg} deals with the trivial case of a virtual fundamental class, namely,  the case of a local complete intersection. In the classical case with values in the Chow groups, the virtual fundamental class is just the cycle class associated to the local complete intersection, in other words, the result given by classical intersection theory. The refined version of this trivial case is still interesting, as it points out how the classical intersection multiplicity is replaced by the $\A^1$-local degree.
\end{remark}

As an example of the above construction we have the virtual fundamental class of the critical locus of a function $f:M\to \A^1$, $M$ a smooth $k$-scheme. The critical locus $Z$ of $f$ is simply the 0-subscheme of the section $df$ of $\Omega_{M/k}$.

Taking the Hessian matrix of  $f$ gives the globally defined morphism of sheaves
\[
H:\sO_Z\to \sHom(\Omega^{\vee}_{M/k}\otimes\sO_Z, \Omega_{M/k}\otimes\sO_Z).
\]
In a coordinate neighborhood with coordinates $x_1,\ldots, x_n$, $H(1)$ is the map sending $\del/\del x_i\otimes f$ to $\sum_j dx_j\otimes \del^2f/\del x_i\del x_j$. We have as well the commutative diagram of sheaves on $Z$
\[
\xymatrix{
\Omega^{\vee}_{M/k}\otimes\sO_Z\ar[r]^{H}\ar[d]_{\ev_{df}}&\Omega_{M/k}\otimes\sO_Z\ar@{=}[d]\\
\sI_Z/\sI_Z^2\ar[r]_d&\Omega_{M/k}\otimes\sO_Z
}
\]
where $\ev_{df}$ is the map evaluating a vector field on $df$, giving us a perfect obstruction theory on $Z$.

Clearly  $[\Omega^{\vee}_{M/k}\otimes\sO_Z\xrightarrow{H} \Omega_{M/k}\otimes\sO_Z]$ has virtual rank 0 and virtual determinant $\omega_{M/k}^{\otimes 2}$.   Thus, this perfect obstruction theory   gives us a virtual fundamental class
\[
[Z, \del_f]^\vir\in \sE^\BM_{0,0}(Z)
\]
for any   cohomology theory $\sE\in \SH(k)$. If in addition $Z$ is projective over $k$, and $\sE$ is $\SL$-oriented,  then we have the push-forward map, giving
\[
\Deg_\sE([Z, \del_f]^\vir):=\pi_{Z*}([Z, \del_f]^\vir)\in \sE^{0,0}(\Spec k).
\]
For instance, we may take $\sE$ to be hermitian $K$-theory or Chow-Witt theory, giving 
\[
\Deg_\sE([Z, \del_f]^\vir)\in \GW(k)
\]
with rank the degree of the usual virtual fundamental class.

One can generalize this construction slightly. Assuming that the critical subscheme of $f$ is a disjoint union of components, $Z\amalg Z'$, we can restrict the whole construction to $Z$, giving the virtual fundamental class $[Z, \del_f]^\vir\in \sE^\BM_{0,0}(Z)$.

As a direct consequence of Proposition~\ref{prop:VirDeg} we have the following description of the virtual fundamental class of a zero-dimensional component of the critical locus.

\begin{corollary}\label{cor:VirClassA1Deg} Let $f:\A^n\to \A^1$ be a function and suppose that the critical  subscheme  of $f$ is a disjoint of closed subschemes $Z\amalg Z'$ with $Z$ of pure dimension zero. We have the polynomial map 
\[
\del f=(\del f/\del X_1,\ldots, \del f/\del X_n): \A^n\to \A^n. 
\]
Then
\[
\pi_{Z*}([Z,\del_f]^\vir)=\delta_{\A^1}(\del f, Z_\red)
\]
in $\GW(k)$.
\end{corollary}

 \end{document}